\documentclass[11pt]{amsart}
\usepackage{latexsym,color,amsmath,systeme, amsthm,amssymb,amscd,amsfonts}
\usepackage{graphicx}

\usepackage{hyperref}
\usepackage{pgf,tikz,pgfplots, pgfplotstable}
\pgfplotsset{compat=1.14}
\usepackage{float}
\usepackage{mathrsfs}
\usetikzlibrary{arrows, spy,intersections, pgfplots.fillbetween}

\usepackage{tkz-euclide}
\usepackage[style=english]{csquotes}
\usepackage{enumerate}
\usepackage{lscape}
\usepackage[width=\textwidth]{caption}
\usepackage{subcaption}
\usepackage{amsmath}
\usepackage{amsfonts}
\usepackage{amssymb}
\usepackage{mleftright}
\usepackage{cleveref}
\usepackage{tkz-tab}
\usepackage{siunitx}
\usetikzlibrary{angles, quotes}
\usepackage{mathtools}

\usetikzlibrary{decorations.pathmorphing}
\tikzset{discont/.style={decoration={zigzag,segment length=12pt, amplitude=4pt},decorate}}
\pgfplotsset{compat=1.11}
\usetikzlibrary{intersections,fillbetween}
 \usetikzlibrary{decorations.text}
\usetikzlibrary{shapes,arrows,intersections, backgrounds}
\usetikzlibrary{scopes,svg.path,shapes.geometric,shadows}
\DeclareCaptionSubType*[arabic]{figure}
\DeclareUnicodeCharacter{2212}{\textendash}
\DeclareUnicodeCharacter{2217}{\textendash}
\DeclareUnicodeCharacter{2032}{\textendash}

\tikzset{
	v1/.style={line width=.5pt,blue!33!black},
	v2/.style={line width=.5pt,blue!66!black},
	v3/.style={line width=.5pt,blue!33},
	v4/.style={line width=.5pt,blue!66},
	v5/.style={line width=.5pt,black}
}
\newcommand{\boundellipse}[3]
{(#1) ellipse (#2 and #3)
}
\definecolor{dred}{HTML}{C11B17}
\definecolor{dgreen}{HTML}{41A317}
\definecolor{dblue}{HTML}{00008B}
\definecolor{niceblue}{HTML}{00008B}
\definecolor{brilliantrose}{rgb}{1.0, 0.33, 0.64}
\definecolor{gold}{HTML}{ffd700}
\definecolor{lgold}{HTML}{ffe140}

\newcommand{\cemph}[1]{\textcolor{dred}{\emph{#1}}}
\newcommand{\opt}{\operatorname{opt}}

\newenvironment{smallpmatrix}{\left( \begin{smallmatrix}}{\end{smallmatrix} \right)}
\newcommand{\smallmat}[1]{\begin{smallpmatrix} #1 \end{smallpmatrix}}

\newcommand{\R}{\mathbb R}
\newcommand{\K}{\mathcal K}

\newcommand{\B}{\mathbb{B}}

\newcommand{\wrt}{w.r.t.}

\newcommand{\conv}{\mathrm{conv}}
\newcommand{\bd}{\mathrm{bd}}
\newcommand{\ext}{\mathrm{ext}}
\newcommand{\pos}{\mathrm{pos}}
\newcommand{\aff}{\mathrm{aff}}

\newcommand{\inter}{\mathrm{int}}
\newcommand{\lin}{\mathrm{lin}}

\newcommand{\norm}[1]{\|#1 \|}
\newcommand{\GH}{{\mathbb{GH}}}
\NewDocumentCommand{\set}{+m+g}
{%
  \IfNoValueTF{#2}
  {%
    \mleft\{#1\mright\}
  }%
  {%
     \mleft\{#1 : #2 \mright\}
  }%
}%

\newcommand{\optc}{\subset^{\operatorname{\opt}}}
\definecolor{zzttqq}{rgb}{0.6,0.2,0}
			\definecolor{ccqqqq}{rgb}{0.8,0,0}
\definecolor{ffvvqq}{rgb}{1,0.3333333333333333,0}
\definecolor{zzffqq}{rgb}{0.6,1,0}
\definecolor{qqwuqq}{rgb}{0,0.39215686274509803,0}
\definecolor{ffzzqq}{rgb}{1,0.6,0}
\definecolor{ffqqqq}{rgb}{1,0,0}
\definecolor{zzttqq}{rgb}{0.6,0.2,0}
\definecolor{uuuuuu}{rgb}{0.26666666666666666,0.26666666666666666,0.26666666666666666}

\newtheorem{thm}{Theorem}[section]
\newtheorem{lemma}[thm]{Lemma}

\newtheorem{proposition}[thm]{Proposition}
\newtheorem{cor}[thm]{Corollary}

\usetikzlibrary{calc,intersections}
\begin{document}



\title[Bounding the diameter-width ratio]{Bounding the diameter-width ratio using containment inequalities of means of convex bodies }

\author[K. von Dichter]{Katherina von Dichter}
\author[M. Runge]{Mia Runge}

\begin{abstract}
We completely describe the region of possible values of the diameter-width ratio for planar pseudo-complete sets in dependence of the Minkowski asymmetry.
In order to do this, we focus on the containment inequalities of $K \cap (-K)$ and $\frac{K-K}{2}$ for a Minkowski centered convex compact set $K$, i.e. we define $\tau(K)$ to be the smallest possible factor to cover $K \cap (-K)$ by a rescalation of $\frac{K-K}{2}$ and give the region of the possible values of $\tau(K)$ in the planar case in dependence of the Minkowski asymmetry of $K$. 
\vspace{-1cm}
\end{abstract}



\maketitle
\section{Introduction and Notation}
We denote the family of \cemph{convex bodies} (full-dimensional compact convex sets) by $\K^n$. Any set $A \in \K^n$ fulfilling $A = t - A$ for some $t \in \R^n$ is called  \cemph{symmetric} and \cemph{0-symmetric} if $t=0$. The family of 0-symmetric bodies is denoted by $\K^n_0$. We write $K \subset_t C$, if there exists $t \in \R^n$, such that $K \subset C + t $. For $K,C \in \K^n$ 
the \cemph{circumradius} and the \cemph{inradius} of $K$ w.r.t.~$C$ 
are defined as $ R(K,C):=\inf\{\rho>0: K \subset_t \rho C\} \quad \text{and} \quad r(K,C) := \sup\{\rho > 0 : \rho C \subset_t  K\},$ while the \cemph{width} and the \cemph{diameter} 
are defined as $ w(K,C):= 2 r(K-K, C-C) \quad \text{and} \quad D(K,C):= 2 R(K-K, C-C).$
We say $K$ is of \cemph{constant width} \wrt~$C$, if $D(K,C)=w(K,C)$, while $K$ is \cemph{complete} \wrt~$C$ if $D(K',C)>D(K,C)$ for every $K' \supsetneq K$.

Bodies of constant width and completeness have been studied in numerous works, including \cite{Gro} (in euclidean spaces); \cite{Eg2, Eg4, MoSch} (in general Minkowski spaces). Constant width always implies completeness and in some cases, for instance, when $C$ is the euclidean ball or a planar set, completeness and constant width always coincide. Such gauges $C$ are called \cemph{perfect}. However, for every $n \ge 3$ there exist (symmetric) $C \in  \K^n$ such that completeness and constant width do not always coincide \cite{Eg2}. Characterizing perfect gauges is still a major open question in convex geometry (see e.g.~\cite{Eg2,MoSch}). More recently in \cite{Ri}, it is shown that for any $K\in \K^n$ complete with respect to $C \in \K^n_0$ holds $\frac{D(K,C)}{w(K,C)} \leq \frac{n+1}{2}$ and that this bound is sharp, if $n$ is odd and $K$ a simplex. This upper bound is improved in \cite{BDG2} by means of the \cemph{Minkowski asymmetry} of $K$. It is defined as 
\[
s(K):=\inf \{ \rho >0: K-c \subset \rho (c-K), \ c \in \R^n \},
\]
and a \cemph{Minkowski center} of $K$ is any $c \in \R^n$ such that $K-c \subset s(K)(c-K)$ \cite{BrG2}. If $0$ is a Minkowski center, we say $K$ is \cemph{Minkowski centered}. Note that $s(K) \in [1,n]$ for $K \in \K^n$; $s(K)=1$ if and only if $K$ is symmetric, and $s(K)=n$ if and only if $K$ is a simplex \cite{Gr}. The improved bound is given as $\frac{D(K,C)}{w(K,C)} \leq \frac{s(K)+1}{2}$.  Also sharpness for $n>2$ odd and any prescribed $s(K) \in [1,n]$ and for $n>2$ even and any $s(K) \in [1,n-1]$ is shown.

We say that $c\in\R^n$ is an \cemph{incenter} of $K$ with respect to $C$ if $c+r(K,C)C\subset K$. In \cite{MoSch} it is shown that all complete bodies satisfy $R(K,C)+r(K,C)=D(K,C)$, while in \cite{BrG2} that the latter condition is equivalent to $\frac{K-K}2 \subset \frac{D(K,C)}2 C \subset \frac{s(K)+1}2 (K-c) \cap (-(K-c))$ for every incenter $c$ of $K$.
This is especially valuable if one understands it as a necessary condition for the gauges $C$ such that a given body $K$ may be complete \wrt~$C$ (keeping in mind that constant width is equivalent to $\frac{K-K}2 = \frac{D(K,C)}2 C$ if $C$ is 0-symmetric). We call
$K \in \K^n$ that satisfies $r(K,C)+R(K,C)=D(K,C)$ to be \cemph{pseudo-complete} w.r.t.~$C$.

We now present the main result of this work. Let $\varphi=\frac{1+\sqrt{5}}{2}$ denote the \cemph{golden ratio} and 
\[ 
c(s) :=
    \begin{dcases}
     1 & s \leq \varphi, \\
    \frac{(s^2+1)^2}{(s^2-1)\left(s^2+2s-1 +2\sqrt{s(s^2-1)}\right)} & \varphi <s \leq \hat s, \\
    \frac{2(s^2-2s-1)}{(s-3)(s+1)}  & \hat s < s \leq 2,  \\
    \end{dcases}
\]
where $\hat s \approx 1.8536$ is such that $\frac{(\hat s^2+1)^2}{(\hat s^2-1)\left(\hat s^2+2\hat s-1 +2\sqrt{\hat s(\hat s^2-1)}\right)}= \frac{2(\hat s^2-2\hat s-1)}{(\hat s-3)(\hat s+1)}$.


\begin{thm}\label{thm:results_pscomp}
Let $K \in \K^2$, $C \in \K^2_0$ and $K$ be pseudo-complete w.r.t. $C$. Then 
\[ 
\frac{D(K,C)}{w(K,C)} \leq \frac{s(K)+1}2 c(s(K)) \leq \frac{\varphi+1}{2}\approx 1.31,
\]
$\frac{D(K,C)}{w(K,C)}=\frac{\varphi+1}{2}$ holds
if and only if $L(K)=\mathbb{GH}:=\conv\left(\left\{\begin{pmatrix} \pm 1 \\ 0\end{pmatrix}, \begin{pmatrix} \pm 1 \\ -1\end{pmatrix}, \begin{pmatrix} 0 \\ \varphi \end{pmatrix} \right\} \right) $
is the \cemph{golden house} and $L$ is some linear transformation.

Moreover, for every pair $(\rho,s)$ with $1\leq s \leq 2$ and $1 \leq \rho \leq \frac{s+1}2 c(s)$, there exists some Minkowski centered $K$, s.t.~$s(K)=s$, and a convex body $C$, s.t.~$K$ is pseudo-complete w.r.t. $C$ and $\frac{D(K,C)}{w(K,C)}=\rho$ (c.f.~\Cref{fig:Dwratio}).
\end{thm}

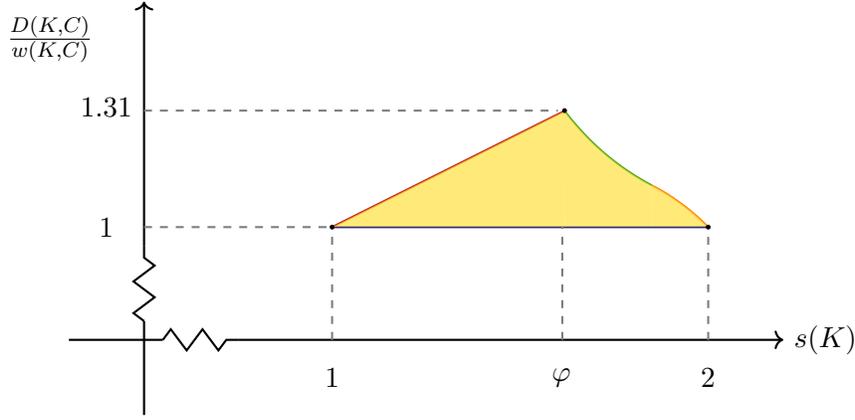
\begin{figure}[ht]
  \begin{tikzpicture}[scale=5]
    \draw[thick, discont] (0.05,0) -- (0.25,0);
    \draw[thick, discont] (0,0.05) -- (0,0.25);
    \draw [thick] (-0.2,0) -- (0.05,0);
    \draw [thick] (0,-0.2) -- (0,0.05);
    \draw[->] [thick] (0.25, 0) -- (1.7, 0) node[right] {$s(K)$};
    \draw[->] [thick] (0, 0.25) -- (0, 0.9);

    \draw [name path=F1,thick, dblue, shift={(-0.5,-0.7)}] (1,1) -- (1.6180339,1);
    \draw [name path=F2,thick, dblue, shift={(-0.5,-0.7)}] (1.6180339,1) -- (1.85,1);
    \draw [name path=F3,thick, dblue, shift={(-0.5,-0.7)}] (1.85,1) -- (2,1);
    
    \draw[name path=F4, thick, dred, domain=1:1.619, smooth, variable=\x, dred, ,shift={(-0.5,-0.7)}]  plot ({\x}, {(\x+1)/2});

    \draw [thick, dashed,gray] (0,0.3) -- (0.5,0.3);
    \draw [thick, dashed,gray] (1.5,0) -- (1.5,0.3);
    \draw [thick, dashed,gray] (0.5,0) -- (0.5,0.3);
    \draw [thick, dashed,gray] (0,0.61) -- (1.1,0.61);
    \draw [thick, dashed,gray] (1.112,0) -- (1.112,0.3);
    
  \draw[name path=F5, domain=1.6180339:1.85, smooth, variable=\x,shift={(-0.5,-0.7)},thick,dgreen] plot ({\x}, {((((\x)^2+1)^2))/(((\x)^2-1)*((\x)^2+2*\x-1+2*sqrt(\x*((\x)^2-1)))) * (\x+1)/2});
   \draw[name path=F6, domain=1.85: 2, smooth, variable=\x,shift={(-0.5,-0.7)},thick, orange] plot ({\x}, {((\x)^2-2*(\x)-1)/(\x-3)});
   \tikzfillbetween[of=F1 and F4,on layer=main]{lgold, opacity=0.7};
    \tikzfillbetween[of=F2 and F5,on layer=main]{lgold, opacity=0.7};
    \tikzfillbetween[of=F3 and F6,on layer=main]{lgold, opacity=0.7};
    \draw (-0.25,0.8) node {$\frac{D(K,C)}{w(K,C)}$};
    \draw (-0.1,0.3) node {$1$};
    \draw (-0.1,0.62) node {$1.31$};
    \draw (0.5,-0.1) node {$1$};
    \draw (1.5,-0.1) node {$2$};
    \draw (1.11,-0.1) node {$\varphi$};
    
    \draw [fill,shift={(-0.5,-0.7)}] (1,1) circle [radius=0.005];
    \draw [fill,shift={(-0.5,-0.7)}] (2,1) circle [radius=0.005];
    \draw [fill,shift={(-0.5,-0.7)}] (1.6180339,1.31) circle [radius=0.005];
     \end{tikzpicture}
   \caption{Region of all possible values for the diameter-width ratio for pseudo complete sets $K$ in dependence of their Minkowski asymmetry $s(K)$: $\frac{D(K,C)}{w(K,C)} \geq 1$ (blue); $\frac{D(K,C)}{w(K,C)} \leq \frac{s(K)+1}{2} c(s(K))$ (red, green and orange, resp.).
   }\label{fig:Dwratio}
\end{figure}

For $K \in \K^n_0$ it is reasonable to inquire $K \cap (-K) = K = \conv(K \cup (-K))$, a condition satisfied only when $0$ acts as the center of symmetry for $K$. 
For this reason in the following we fix the center of the convex body $K$ to be the Minkowski center. 
Furthermore, if $K$ is not 0-symmetric, $K \cap (-K)$ can become arbitrarily small. 

The proof of \Cref{thm:results_pscomp} is heavily based on the containment inequalities of $K \cap (-K)$ and $\frac{K-K}{2}$ for a Minkowski centered convex compact set $K$. 

We define $\tau(K)>0$ such that 
\[
K  \cap (-K) \subset^{\opt} \tau(K) \, \left( \frac{K-K}{2} \right). 
\]

We describe the possible $\tau$ values of $K$ in dependence of its asymmetry (c.f.~\Cref{fig:alpha-region}). This result will be the biggest milestone for showing \Cref{thm:results_pscomp}. 
\begin{thm}\label{thm:PlanarCase}
Let $K \in\K^2$ be Minkowski centered. Then 
\[
\frac{2}{s(K)+1} \leq \tau(K) \leq c(s(K)). 
\] 
 
Moreover, for every pair
$(\tau,s)$, such that $1 \leq s \leq 2$ and $\frac{2}{s+1} \leq \tau \leq c(s)$, there exists a Minkowski centered $K_{s,\tau} \in\K^2$, such that $s(K_{s,\tau})=s$ and $\tau(K_{s,\tau})=\tau$.
\end{thm}

\section{Definitions and Propositions}
For any $X,Y \subset\R^n$, $\rho \in \R$ let $X+Y =\{x+y:x\in X,y\in Y\}$ be the \cemph{Minkowski sum} of $X$, $Y$ and  $\rho X= \{ \rho x: x \in X\}$ the \cemph{$\rho$-dilatation} of $X$, and abbreviate $(-1)X$ by $-X$. 
For any $X\subset\R^n$ let $\conv(X), \pos(X)$, $\lin(X)$, and $\aff(X)$ denote the \cemph{convex hull}, the \cemph{positive hull}, the \cemph{linear hull}, and the \cemph{affine hull} of $X$, respectively. 
A \cemph{segment} is the convex hull of $\{x,y\} \subset \R^n$, which we abbreviate by $[x,y]$. With $u^1, \dots, u^n$ we denote the \cemph{standard unit vectors} of $\R^n$. For every $X \subset \R^n$ let $\bd(X)$ and $\inter(X)$ denote the \cemph{boundary} and \cemph{interior} of $X$, respectively. Let us denote the \cemph{Euclidean norm} of $x\in\R^n$ by $\|x\|$, the \cemph{Euclidean unit ball} by $\B_2=\{x \in\R^n : \|x\|\leq 1\}$. In generalized Minkowski spaces the norm is replaced by a \cemph{gauge function},
i.e.~a function of the type $\norm{\cdot}_C \colon \R^n \to [0,\infty), \norm{x}_C = \inf \{ \rho \geq 0 : x \in \rho C \}$ with a (not necessarily symmetric) gauge body $C \in  \K^n$ satisfying $0 \in \text{int}(C)$. 
In case $u^1,\dots,u^{n+1}\in\R^n$ are affinely independent, we say that $\conv(\{u^1,\dots,u^{n+1}\})$ is an \cemph{$n$-simplex}. 
For any $a\in\R^n \setminus \{0\}$
and $\rho\in\R$, let $H^{\le}_{a,\rho} = \{x\in\R^n: a^Tx \leq \rho\}$ denote a \cemph{halfspace} with its boundary  being the \cemph{hyperplane} $H_{\left(a,\rho\right)}^==\{x\in\R^n: a^Tx = \rho\}$. 
Analogously, we define $H_{\left(a,\rho\right)}^\ge, H_{\left(a,\rho\right)}^<,H_{\left(a,\rho\right)}^>$. 
We say that the halfspace $H^{\le}_{\left(a,\rho\right)}$ \cemph{supports} $K \in\K^n$ in $q \in K$, if $K \subset H^{\le}_{\left(a,\rho\right)}$ and $q \in H_{\left(a,\rho\right)}^=$. We denote the set of all \cemph{extreme points} of $K$ by $\ext(K)$. 
For any $K, C\in\K^n$ we say $K$ is \cemph{optimally contained} in $C$, and denote it by $K\subset^{\opt}C$, if $K\subset C$ and $K\not \subset_t \rho C$ for any $\rho \in [0,1)$.

We start off with a very basic but useful fact.
\begin{lemma}\label{lem:sixareas}

Let the plane $\R^2$ be split into six areas by three pairwise different lines passing through the origin and each of the points $p_1,p_2,p_3$ be contained in a different of those areas, any two which are non-consecutive. Then $0\in\conv\left(\set{p^1,p^2,p^3}\right)$.
\end{lemma}
\begin{proof}
    Assume that $0\notin\conv\left(\set{p^1,p^2,p^3}\right)$. We can separate $\set{0}$ and $\conv\left(\set{p^1,p^2,p^3}\right)$, thus there exists $a\in\R^2$ such that $a^Tp^i<0$ for $i\in\set{1,2,3}$. But since $H^{\leq}_{(a,0)}$ can only intersect at most four of the six areas, this contradicts the fact that $p^i$, $i\in\set{1,2,3}$, are contained in non-consecutive areas.
\end{proof}

We recall the characterization of the optimal containment under homothety in terms of the touching conditions \cite[Theorem 2.3]{BrKo}. 
\begin{proposition}\label{prop:Opt_Containment}
Let $K,C\in\mathcal K^n$ and $K\subset C$. Then the following are equivalent:
\begin{enumerate}[(i)]
\item $K\subset^{\opt}C$
\item There exist $k\in\{2,\dots,n+1\}$, $p^j\in K\cap \bd(C)$, $j=1,\dots,k$, and $a^j$ outer normals of supporting
halfspaces of $K$ and $C$ at $p^j$,
such that $0\in\conv(\{a^1,\dots,a^k\})$.
\end{enumerate}
\end{proposition}

In the planar case \Cref{prop:Opt_Containment} 
implies the following corollary (c.f.~\cite[Corollary 8]{BDG}). 

\begin{cor}\label{cor:zero-inside}
Let $K \in \K^2$ with $s(K) > 1$. Then $K$ is Minkowski centered if and only if there exist $p^1,p^2,p^3\in\bd(K)\cap(-\frac{1}{s(K)} K)$ and $a^i$, $i=1,2,3$, outer normals of supporting halfspaces of $K$ in $p^i$,
such that $0 \in \inter(\conv(\{p^1,p^2,p^3\}))$ and $0\in\conv(\{a^1,a^2,a^3\})$.
\end{cor}
\begin{proof}
    There exist two or three such touching points in $\bd(C)\cap (-\frac{1}{s(C)}\bd(C))$ as described in \Cref{prop:Opt_Containment}. If it were only two it would follow that $s(C)=1$. 
    Now let $S$ be the intersection of the three supporting hyperplanes at the touching points $\set{p^1,p^2,p^3}$. In addition, parallel hyperplanes support $-\frac{1}{s(C)}C$ at the opposing points $-\tfrac{1}{s(C)}p^i$. Thus, $\conv(\set{p^1,p^2,p^3})$ is contained in the symmetric hexagon $S\cap (-S)$. The lines through opposing vertices of this hexagon form a situation as in \Cref{lem:sixareas} and therefore $0 \in \conv(\{p^1,p^2,p^3\})$. Since $0\in \bd (\conv(\{p^1,p^2,p^3\})) $ would imply $s(C)=1$, we obtain $0 \in \inter(\conv(\{p^1,p^2,p^3\}))$. 
\end{proof}

Let $\alpha(K):=R\left(K  \cap (-K), \conv(K \cup (-K))\right)$ and recall \cite[Lemma 2.2, Corollary 2.3]{BDG1}. 
\begin{proposition}\label{prop:old_results}
Let $K \in\K^n$ be Minkowski centered and let $L$ be a regular linear transformation. Then $\alpha(K)=\alpha(L(K))$ and $\tau(K)=\tau(L(K))$. 
\end{proposition}

We recall another result from \cite{BDG} based on \Cref{prop:old_results} about the equality cases for the upper bound in the inequality $\alpha(K) \le 1$ in the planar case.

\begin{proposition}\label{prop:GoldenHouse}
Let $K\in\mathcal K^2$ be Minkowski centered such that $\alpha(K)=1$ or $\tau(K)=1$. Then $s(K)\leq \varphi$. Moreover, if $s(K)=\varphi$, there exists a linear transformation $L$ such that $L(K)=\mathbb{GH}=\conv\left(\left\{\begin{pmatrix} \pm 1 \\ 0\end{pmatrix}, \begin{pmatrix} \pm 1 \\ -1\end{pmatrix}, \begin{pmatrix} 0 \\ \varphi \end{pmatrix} \right\} \right) $ is the golden house.
\end{proposition}

For $K \in \K^n$ Minkowski centered we call any $p \in \bd(K) \cap \bd\left(-\frac{1}{s(K)}K\right)$ an \cemph{asymmetry point} of $K$, and any triple of asymmetry points, with the properties as stated in \Cref{cor:zero-inside}, to be \cemph{well-spread}. For $K \in \K^2$ we call $z^1, z^2 \in \bd(K) \cap \bd(-K)$ \cemph{consecutive}, if there exists no point $z \in \bd(K) \cap \bd(-K) \cap \inter (\pos \{z^1, z^2\})$.

While the first two statements of the following proposition combine 
\cite[Theorem 1.2]{BDG2} and \cite[Lemma 3.1]{BDG2}, the third one is added here for the sake of completeness.
\begin{proposition}\label{prop:crossings}   
Let $K \in\K^2$ be Minkowski centered with $s(K) \geq \varphi$ and let $z^1, z^2, z^3 \in \bd(K) \cap \bd(-K)$ be such that $z^1$,$z^2$ and $z^2,z^3$ are consecutive. Then the following holds: 
\begin{enumerate}[(i)]
\item the set
$\bd(K) \cap \bd(-K)$ consists of exactly 6 points, 
\item there exists either an asymmetry point of $K$ or of $-K$ in $\pos\{z^1,z^2\}$,
\item if there is an asymmetry point of $K$ in $\pos\{z^1,z^2\}$, there is no asymmetry point of $K$ in $\pos\{z^2,z^3\}$.
\end{enumerate}
\end{proposition}

\begin{proof}[Proof of \Cref{prop:crossings} (iii)]
Let $z^1, z^2, z^3 \in \bd(K) \cap \bd(-K)$ be consecutive and $p^1\in\pos\{z^1,z^2\}$, $p^2\in\pos\{z^1,z^2\}$ be two asymmetry points.

Assume we have three consecutive points $z^1, z^2, z^3 \in \bd(K) \cap \bd(-K)$ such that $p^1\in\pos\{z^2,z^3\}$, $p^2\in\pos\{z^1,z^2\}$ are asymmetry points of $K$. 
Since $sp^1, sp^2 \in \bd(-K)$, for $x\in\pos\set{p^1,p^2}$ holds $\norm{x}_K\geq \norm{x}_{-K}$. 
But since $z^2 \in \bd(-K) \cap \bd(_K) \in \pos\set{p^1,p^2}$, there is a hyperplane supporting both $K$ and $-K$ at $z^2$.

This implies $\alpha(K)=1$ because $z^2$ is contained in both the boundary of $\conv(K\cup(-K))$ and $K\cap(-K)$ and therefore $s(K)\leq \varphi$. If $\alpha(K)=1$ and  $s(K)=\varphi$, we know by \Cref{prop:GoldenHouse} that $K$ is a linear transformation of the golden house for which the claim holds. 
By \Cref{lem:sixareas}, the asymmetry points are also well-spread. 
\end{proof}

\section{The Proof of Theorem 1.2}
In this section we prove \Cref{thm:PlanarCase} which describes the containment factor between the minimum $K\cap(-K)$ and the arithmetic mean $\frac{K-K}{2}$ of a planar Minkowski-centered convex body $K$. We begin with a technical lemma which shows that we can transform a general $K$ with $s(K)>\varphi$ into a heptagon with the same asymmetry and larger or equal $\tau$-value.

\begin{lemma}\label{lem:beforePlanarCase}
Let $K\in\mathcal K^2$ be Minkowski centered with $s(K)>\varphi$. Then there exist a point $p \in \bd(K) \cap \bd(-K)$ and $p^1,p^2,p^3 \in \R^2$, such that 
\begin{equation*}
 \bar K=\conv(\{\pm p, p^2,p^3,-s(K)p^1,-s(K)p^2, -s(K)p^3\})
\end{equation*}
is Minkowski centered, $p^1,p^2,p^3 \in \R^2$ form a triple of well-spread asymmetry points of $\bar K$, $p^1 =-\gamma d^1$ for some 
$0< \gamma < 1$, where $d^1=\aff (\{-p,p^3\}) \cap \aff (\{p,p^2\})$, and the following properties are fullfilled: 
\begin{enumerate}[(i)]
  \item $s(\bar K)=s(K)$, 
  \item $\tau(\bar K) \geq \tau(K)$, and
  \item $p \in \bd(\bar K \cap (-\bar K)) \cap \tau(\bar K) \bd \left(\frac{\bar K-\bar K}{2}\right)$.
\end{enumerate}

\end{lemma}

\begin{proof} 
The proof is split into two parts. The goal is to transform a general convex set $K$ into a convex body with an "simplier shape" while keeping the asymmetry and preserving the needing properties. In $a)$ we show that we can cut off everything except the intersection $K\cap(-K)$ and the asymmetry points $-sp^i$, $i\in\set{1,2,3}$, without changing the asymmetry or decreasing the value of $\tau$. In $b)$, we further transform the convex body into a certain heptagon. 
\begin{enumerate}[a)] \item 
We show that there exist a triple of well-spread asymmetry points of $K$ $p^1,p^2,p^3$, such that 
\begin{equation}\label{eq:K'}
K':=\conv ((K \cap (-K)) \cup \{-s(K)p^1, -s(K)p^2, -s(K)p^3 \})     
\end{equation}
is Minkowski centered with
\begin{enumerate}[(i)]
  \item $s(K')=s(K)$, 
  \item $\tau(K') \geq \tau(K)$, and
  \item there exists $p \in \bd(K' \cap (-K')) \cap \tau(K') \bd \left(\frac{K'-K'}{2}\right)$ with $p \in \bd(K') \cap \bd(-K')$.
\end{enumerate}
\vspace{3mm}
 Let $s:=s(K)$. Since $s>\varphi$, by \Cref{prop:crossings} $\bd(K) \cap \bd(-K)$ consists of exactly 6 points and there exist consecutive points $z^{i,i+1}, z^{i,i+2} \in \bd(K) \cap \bd(-K)$  with $-p^i \in \pos \{z^{i,i+1}, z^{i,i+2}\}$, $i=1,2,3$ for well-spread asymmetry points $p^1,p^2,p^3$ of $K$.
Indices are to be understood modulo 3.

We show the properties (i)-(iii) for the set $K'$. 

Note that $s(K')=s(K)$. This can be easily seen from the fact that $\frac{1}{s}(K\cap(-K))\subset -K'$ and $\set{-p^1,-p^2,-p^3}\subset (K\cap(-K))\subset-K'$. The points $p^1,p^2,p^3$ are still well-spread asymmetry points of $K'$. Note that the boundary of $-K'$ does not intersect the segments $[p^i, z^{i,j}]$, $i=1,2,3, j=i+1,i+2$, in the boundary of $K'$. Otherwise, we would have more than six intersection points between the boundaries of $K'$ and $-K'$, contradicting \Cref{prop:crossings}. 

Observe that $K'\cap(-K') = K \cap(-K)$ together with $K' \subset K$ implies $\tau(K')\geq \tau(K)$.
\begin{figure}[ht] 
\centering
\begin{tikzpicture}[scale=5]
\draw [thick] plot [smooth cycle] coordinates {(-0.1,0.45) (0.5,-0.37) (1.7,0.3)};
\draw [thick, gray, rotate around={180:(0.68,0.12)}] plot [smooth cycle] coordinates {(-0.1,0.45) (0.5,-0.37) (1.7,0.3)};

\draw [thick, dotted] (0.68,0.12) -- (1.69,0.29);
\draw [thick, dotted, rotate around={180:(0.68,0.12)}] (0.68,0.12) -- (1.69,0.29);
\draw [thick, dotted] (0.68,0.12) -- (0.47,-0.37);
\draw [thick, dotted, rotate around={180:(0.68,0.12)}] (0.68,0.12) -- (0.47,-0.37);
\draw [thick, dotted] (0.68,0.12) -- (-0.12,0.41);
\draw [thick, dotted, rotate around={180:(0.68,0.12)}] (0.68,0.12) -- (-0.12,0.41);


\fill [fill=lgold, fill opacity=0.7, domain=1:1.3, variable=\x, rotate around={180:(0.68,0.12)}] (0.42,0.47)--(0.89,0.6) -- (1.17,0.41)--(1.3,0.24)--(1.43,0.04)--(1.485,-0.17)--(0.93,-0.232)--(0.4,-0.2)--(0.2,-0.17)--(-0.35,-0.05)--(-0.07,0.2)--(0,0.25);

\fill [fill=lgold, fill opacity=0.7, domain=1:1.3, variable=\x] (0.42,0.47)--(0.89,0.6) -- (1.17,0.41)--(1.3,0.24)--(1.43,0.04)--(1.485,-0.17)--(0.93,-0.232)--(0.4,-0.2)--(0.2,-0.17)--(-0.35,-0.05)--(-0.07,0.2)--(0,0.25);

\draw [thick] (1.15,0.42) -- (1.7,0.29) -- (1.43,0.04);
\draw [thick, rotate around={180:(0.68,0.12)}] (0.92,-0.24) -- (1.49,-0.17) -- (1.43,0.04);
\draw [thick, rotate around={180:(0.68,0.12)}] (0.42,0.47) --(0.89,0.6) -- (1.14,0.43);

\draw [thick, rotate around={180:(0.68,0.12)}] (1.15,0.42) -- (1.7,0.29) -- (1.43,0.04);
\draw [thick] (0.92,-0.24) -- (1.49,-0.17) -- (1.43,0.04);
\draw [thick] (0.42,0.47) --(0.89,0.6) -- (1.14,0.43);

\draw [dred, thick] (1.02,0.51) -- (1.31,0.48)--(1.43,0.36); 
\draw [dred, thick] (1.565, 0.17) -- (1.6,0.04)--(1.45,-0.08); 
\draw [dred, thick] (1.232, -0.202) -- (1,-0.28)--(0.7, -0.3); 
\draw [dred, rotate around={180:(0.68,0.12)}, thick] (1.02,0.51) -- (1.31,0.48)--(1.43,0.36); 
\draw [dred, rotate around={180:(0.68,0.12)}, thick] (1.565, 0.17) -- (1.6,0.04)--(1.45,-0.08); 
\draw [dred, rotate around={180:(0.68,0.12)}, thick] (1.232, -0.202) -- (1,-0.28)--(0.7, -0.3); 

\draw [dgreen, thick] plot [smooth] coordinates {(0.35, -0.27) (0.48,-0.29) (0.7, -0.3)};
\draw [dgreen, thick] plot [smooth] coordinates {(1.565, 0.17)(1.5,0.28) (1.44, 0.35)};
\draw [dgreen, thick] plot [smooth] coordinates {(-0.1, 0.31)(-0.01,0.38) (0.13, 0.44)};
\draw [dgreen, rotate around={180:(0.68,0.12)}, thick] plot [smooth] coordinates {(0.35, -0.27) (0.48,-0.29) (0.7, -0.3)};
\draw [dgreen, rotate around={180:(0.68,0.12)}, thick] plot [smooth] coordinates {(1.565, 0.17)(1.5,0.28) (1.44, 0.35)};
\draw [dgreen, rotate around={180:(0.68,0.12)}, thick] plot [smooth] coordinates {(-0.1, 0.31)(-0.01,0.38) (0.13, 0.44)};

\draw (0.59,0.05) node {$0$};
\draw (1.1,0.35) node {$z^{3,1}$};
\draw (1.25,0.05) node {$z^{3,2}=p$};
\draw (0.9,-0.15) node {$z^{1,2}$};
\draw (0.3,-0.1) node {$z^{1,3}$};
\draw (0.05,0.2) node {$z^{2,3}$};
\draw (0.45,0.4) node {$z^{2,1}$};

\draw (0.45,-0.44) node {$-sp^1$};
\draw (-0.22,0.45) node {$-sp^2$};
\draw (1.8,0.33) node {$-sp^3$};
\draw (1.55,-0.2) node {$sp^2$};

\draw (1.6,0.4) node {$K$};
\draw (-0.2,-0.18) node {$-K$};
\draw (1.75,0.05) node {$\frac{K'-K'}{2}$};

\draw [fill] (0.68,0.12) circle [radius=0.01];
\draw [fill] (0.468,-0.363) circle [radius=0.01];
\draw [fill] (-0.125,0.41) circle [radius=0.01];
\draw [fill] (1.7,0.29) circle [radius=0.01];

\draw [fill, rotate around={180:(0.677,0.12)}] (1.7,0.29) circle [radius=0.01];
\draw [fill, rotate around={180:(0.677,0.12)}] (0.468,-0.363) circle [radius=0.01];
\draw [fill, rotate around={180:(0.677,0.12)}] (-0.125,0.41) circle [radius=0.01];

\draw [fill] (1.43,0.04) circle [radius=0.01]; 
\draw [fill,  rotate around={180:(0.68,0.12)}] (1.43,0.04) circle [radius=0.01]; 
\draw [fill] (0.43,0.475) circle [radius=0.01];
\draw [fill,  rotate around={180:(0.68,0.12)}] (0.43,0.475) circle [radius=0.01];
\draw [fill] (1.15,0.42) circle [radius=0.01];
\draw [fill,  rotate around={180:(0.68,0.12)}] (1.15,0.42) circle [radius=0.01];

\draw [fill] (0.35, -0.27) circle [radius=0.01];
\draw [fill] (0.7, -0.3) circle [radius=0.01];
\draw [fill] (1.565, 0.17) circle [radius=0.01];
\draw [fill] (1.44, 0.35) circle [radius=0.01];
\draw [fill] (-0.1, 0.31) circle [radius=0.01];
\draw [fill] (0.13, 0.44) circle [radius=0.01];
\draw [rotate around={180:(0.68,0.12)}, fill] (0.35, -0.27) circle [radius=0.01];
\draw [rotate around={180:(0.68,0.12)}, fill] (0.7, -0.3) circle [radius=0.01];
\draw [rotate around={180:(0.68,0.12)}, fill] (1.565, 0.17) circle [radius=0.01];
\draw [rotate around={180:(0.68,0.12)}, fill] (1.44, 0.35) circle [radius=0.01];
\draw [rotate around={180:(0.68,0.12)}, fill] (-0.102, 0.31) circle [radius=0.01];
\draw [rotate around={180:(0.68,0.12)}, fill] (0.13, 0.44) circle [radius=0.01];

\draw [fill] (1.31,0.48) circle [radius=0.01]; 
\draw [fill] (1.6,0.04) circle [radius=0.01]; 
\draw [fill] (1,-0.28) circle [radius=0.01]; 
\draw [rotate around={180:(0.68,0.12)}, fill] (1.31,0.48) circle [radius=0.01]; 
\draw [rotate around={180:(0.68,0.12)}, fill] (1.6,0.04) circle [radius=0.01]; 
\draw [rotate around={180:(0.68,0.12)}, fill] (1,-0.28) circle [radius=0.01]; 
\end{tikzpicture}
\caption{Construction from \Cref{lem:PlanarCase}: 
$K$ (black), $-K$ (gray), 
$K'$,$-K'$ (yellow), $\bd \left(\frac{K'-K'}{2} \right)$ (red (segments) and green (possibly arcs)). 
}
\label{fig:K'}
\end{figure}
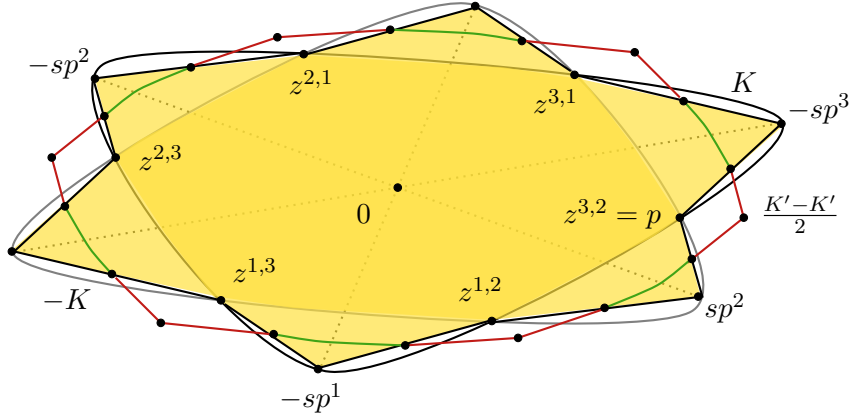

Note that the arithmetic mean $\frac{K'-K'}{2}$ is the convex hull of all midpoints between one extreme point of $K'$ and one of $-K'$.

We define $p \in \bd(K) \cap \bd(-K) \cap \pos\left( \set{-p^3,p^2}\right)$. Then $0$ and $-sp^3$ belong to different halfspaces w.r.t.~the hyperplane $\aff(\set{p,sp^2})$. Indeed, if this would be not the case, this hyperplane would support both $K'\cap(-K')$ and $\conv(\
K'\cup (-K'))$, implying $\alpha(K')=1$, which contradicts $s(K')>\varphi$. 
A similar statement holds for all $\pm sp^i$. 

Note also that $0$ and $p$ belong to the same halfspace w.r.t.~$\aff(\set{-sp^3,sp^2})$.
Thus, we obtain 

\begin{align*}
\bd\left(\frac{K'-K'}{2}\right) &\cap \pos\left(\left\{ -sp^3,\frac{-sp^3+sp^2}{2}\right\}\right) =
\\
 &\left\{ \frac{x-sp^3}{2}, x \in \bd(K') \cap \pos(\{ -sp^3,p\}) \right\} \cup \left[\frac{-sp^3+p}{2}, \frac{-sp^3+sp^2}{2}\right].
\end{align*}

Note that the boundary of $\frac{K'-K'}{2}$ has a similar structure in any other positive hull of this type and it consists of segments and arcs as shown in \Cref{fig:K'}.

 

For any $v \in \bd(K' \cap (-K'))$ we define $\rho(v)>0$ such that 
\[
\rho(v) v \in \bd\left ( \frac{K'-K'}{2}\right). 
\]
Since by \cite[Theorem 1.3]{BDG2} $\frac{K'-K'}{2}\optc \frac{s+1}{2}K'\cap(-K')$, we have $\rho(v)\leq \frac{s+1}{2}$ for all $v\in \bd (K' \cap (-K'))$. 

Observe that 
\[
(\tau(K'))^{-1}=\min_{v \in \bd(K' \cap (-K'))} \rho(v).  
\]

We now show that for any $v \in \bd(K' \cap (-K')) \cap \pos(\{ -sp^3,\frac{-sp^3+sp^2}{2}\})$ 
\[
\rho(v)\geq \rho(p). 
\]
By construction, $\rho(p)p \in [\frac{-sp^3+p}{2}, \frac{-sp^3+sp^2}{2}] \cup [\frac{sp^2+p}{2}, \frac{-sp^3+sp^2}{2}]$. Consider first $v \in \bd(K' \cap (-K')) \cap \pos(\{ \frac{-sp^3+p}{2}, \frac{-sp^3+sp^2}{2}\})$. 
Observe that the hyperplane $\aff(\{ sp^2,p\})$ supports $K'\cap (-K')$ in $p$, and that $\rho(v)v \in \left[\frac{-sp^3+p}{2}, \frac{-sp^3+sp^2}{2}\right]$, which is a segment parallel to $\aff(\{ sp^2,p\})$. Thus, taking $\rho^* \geq 1$ such that $\rho^* p \in \aff(\set{\frac{-sp^3+p}{2}, \frac{-sp^3+sp^2}{2}})$, we see that 
\[
\rho(v)\geq \rho^*\geq \rho(p). 
\]

\begin{figure}[ht] 
    \centering
    \begin{tikzpicture}[scale=6]
        \tkzDefPoint(0,0){0}
    \tkzDefPoint(1,0){z}
    \tkzDefPoint(1.5,0.75){sp}
    \tkzDrawPolygon[black, thick](0,z,sp);
    \tkzDefPointOnLine[pos=.5](z,sp)
    \tkzGetPoint{q1}
\tkzDefPointOnLine[pos=.55](0,sp)
    \tkzGetPoint{p}
    \tkzDefPointOnLine[pos=.5](p,sp)
    \tkzGetPoint{q2}
    \draw [thick,red] (z) to [out=90, in=-40] (p) ;
    \tkzDrawSegment[thick, black](0,q1)
    \draw [thick, brown] (q1) to [out=90, in=-40] (q2) ;

    \tkzDefPointOnLine[pos=.75](0,q1)
    \tkzGetPoint{v1}
    \tkzDrawPoint[fill=black, radius=0.4pt,label=below left:\small{$v^1$}](v1)
    \tkzDrawSegment[thick](v1,sp)
    \tkzDefPointOnLine[pos=.5](v1,sp)
    \tkzGetPoint{q3}
    \tkzDrawPoint[fill=black,radius=0.4pt](q3)
    \tkzLabelSegment[right](sp,q1){\small{$\rho(v^2)v^2$}};
    \tkzDrawLine[thick, dblue, add = 0.2 and 0.5](z,v1)
    \tkzDrawLine[thick, dblue, add = 0.2 and 0.5](q1,q3)
    \tkzDrawSegment[thick](0,q3)
    \tkzDefPointOnLine[pos=0.715](0,q3)
    \tkzGetPoint{v2}
    \tkzDrawPoint[fill=black,radius=0.4pt,label=below left:\small{$v^2$}](v2)
    \tkzDrawPoint[fill=black,radius=0.4pt,label=left:$0$](0)
    \tkzDrawPoint[fill=black,radius=0.4pt,label=right:$p$](z)
    \tkzDrawPoint[fill=black,radius=0.4pt,label=right:$-sp^3$](sp)
    \tkzDrawPoint[fill=black,radius=0.4pt,label=right:\small{$\frac{-sp^3+p}{2}=\rho(v^1)v^1$}](q1)
    \tkzDrawPoint[fill=black,radius=0.4pt,label=above left:\small{$-p^3$}](p)
    \tkzDrawPoint[fill=black,radius=0.4pt,label=above left:\small{$\frac{s+1}{2}(-p^3)$}](q2)
    \draw (0.95,-0.1) node {$H_1^=$};
    \end{tikzpicture}
    \caption{Construction used in the proof of Part a)(iii) of \Cref{lem:beforePlanarCase}: $\bd(-K')$ (red), $\frac{K'-K'}{2}$ (brown).}
    \label{fig:enter-label}
\end{figure}
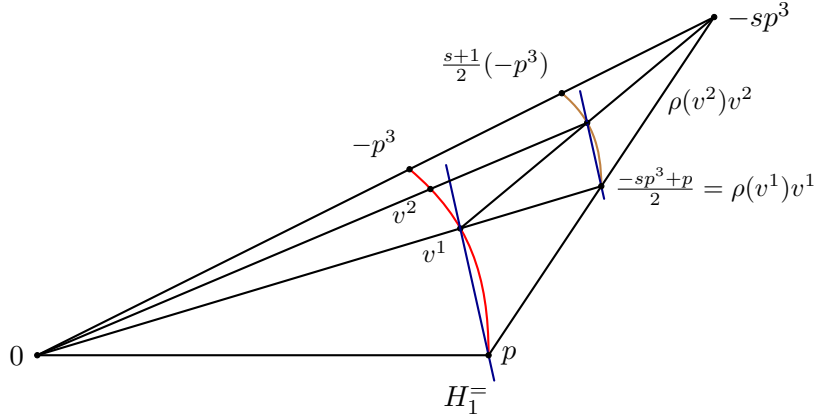

Now let us consider $v \in \bd(K' \cap (-K')) \cap \pos(\{ \frac{-sp^3+p}{2}, -sp^3 \})$.
We define $v^1,v^2\in \bd(K' \cap (-K'))$ such that $\rho(v^1)v^1=\frac{-sp^3+p}{2}$, $\rho(v^2 )v^2=\frac{-sp^3+v^1}{2}$ and $H_1^=:= \aff(\{ v^1,p\})$. Note that $\rho(v^1)\geq\rho(p)$, $\rho(-p^3)=\frac{s+1}{2}\geq \rho(p)$, and the segment $[\rho(v^2) v^2, \rho(v^1) v^1]$ is parallel to $H^=_1$. Denoting by $H_1^\leq$ the halfspace containing the origin defined by $H^=_1$, we have
\[
\bd(K' \cap (-K')) \cap \pos(\{ v^1, -sp^3 \}) \subset H_1^\leq \quad \text{and} \quad [\rho(v^2) v^2, \rho(v^1) v^1] \subset \frac{K'-K'}{2}. 
\]
Thus, for any $v \in \bd(K' \cap (-K')) \cap \pos(\{ v^1, v^2\})$ holds $\rho(v)\geq \rho(v^1)$. 

Since in every step the cone is "cut in half", we can repeat the argument from above iteratively. Thus, by replacing in the next iteration $p,v^1$ with $v^1, v^2$ and so on, one can show that for any $v \in \bd(K' \cap (-K')) \cap \pos(\{ -sp^3, p\})$ holds $\rho(v)\geq \rho(p)$. Thus, one of the intersection points $z^{i,j}$ has the smallest distance to the boundary of $\frac{K'-K'}{2}$ and there exists $p \in \bd(K' \cap (-K')) \cap \tau(K') \bd \left(\frac{K'-K'}{2}\right)$ with $p \in \bd(K') \cap \bd(-K')$, as desired.


\vspace{0.5cm}

\item 
We proceed with the following transformations, defining successively the bodies $K'_i$, $1 \leq i \leq 4$ and showing in each step the desired properties of those sets. 

For any $v \in \bd(K'_i \cap (-K'_i))$ we define $\rho_i(v)>0$ such that 
\[
\rho_i(v) v \in \bd\left ( \frac{K'_i-K'_i}{2}\right),  \quad 1 \leq i \leq 4. 
\]

\vspace{3mm}
\begin{enumerate}[Step 1:]
\item Let $\gamma_i \leq 1$ be such that  $\gamma_i p^i \in [z^{i+1,i}, -z^{i+2,i}]$ for $i=1,2,3$. Replace $K'$ by 
\[
K'_1=\conv\left(\{-\gamma_i sp^i, z^{i,i+1},z^{i,i+2}: i=1,2,3\}\right),
\]
as in \Cref{fig:step1}. 

We show that for $K'_1$ holds 
\begin{enumerate}[(i)]
  \item $s(K'_1)=s(K')$, 
  \item $\tau(K'_1) \geq \tau(K')$, and
   \item there exists $p \in (\bd(K'_1) \cap \bd(-K'_1)) \cap \tau(K'_1) \bd \left(\frac{K'_1-K'_1}{2}\right)$.
\end{enumerate}
\end{enumerate}

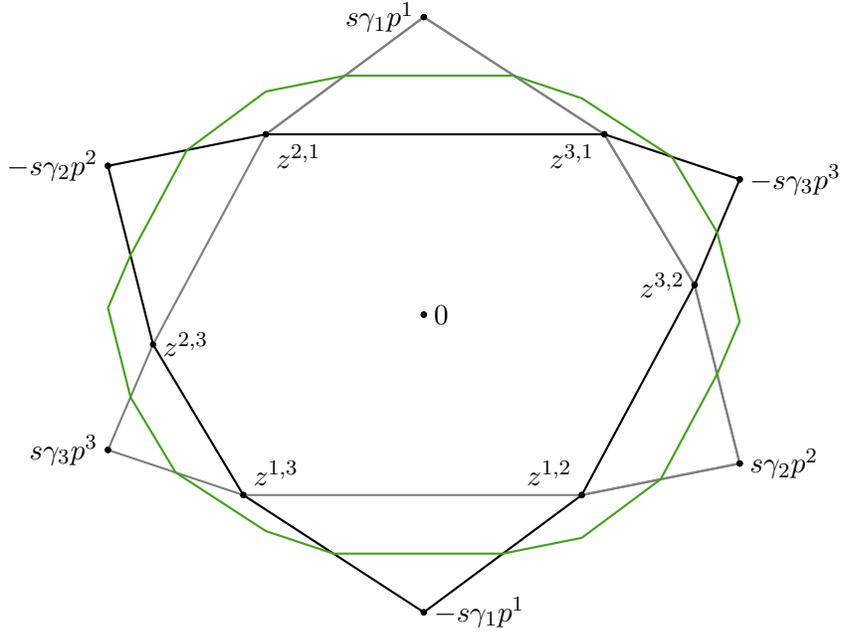
\begin{figure}[ht]
\centering
\scalebox{1}{
 \begin{tikzpicture}[scale=6]
\tkzDefPoints{0/0/zero, 0/.66/sp1, .7/0.3/msp3, 0.4/0.4/z1,-0.4/-0.4/mz1,  -0.35/0.4/z2, 0.35/-0.4/mz2,  0.6/0.066/p, 0.7/-.33/sp2, 0/-0.66/msp1, -0.7/-0.3/sp3, -0.6/-0.066/mp, -0.7/.33/msp2}
\tkzDrawPoints[fill=black](zero,sp1,msp3,sp2,z1,z2,p,mz2,mz1,msp1,sp3,msp2)
\tkzDrawSegments[thick, gray](sp1,z1 z1,p p,sp2 sp2,mz2 sp1,z2)
\tkzDrawPolygon[thick](z2,z1,msp3,p,mz2,msp1,mz1,mp,msp2)
\tkzDrawPolygon[thick, gray](mz2,mz1,sp3,mp,z2,sp1,z1,p,sp2)

\tkzDefMidPoint(sp1,z1) \tkzGetPoint{y1}
\tkzDefMidPoint(sp1,msp3) \tkzGetPoint{y2}
\tkzDefMidPoint(msp3,z1) \tkzGetPoint{y3}
\tkzDefMidPoint(msp3,p) \tkzGetPoint{y4}
\tkzDefMidPoint(msp3,sp2) \tkzGetPoint{y5}
\tkzDefMidPoint(p,sp2) \tkzGetPoint{y6}
\tkzDefMidPoint(mz2,sp2) \tkzGetPoint{y7}
\tkzDefMidPoint(sp2,msp1)\tkzGetPoint{y71}
\tkzDefMidPoint(mz2,msp1)\tkzGetPoint{y72}
\tkzDefMidPoint(msp1,mz2) \tkzGetPoint{y8}
\tkzDefMidPoint(msp1,mz1) \tkzGetPoint{y9}
\tkzDefMidPoint(msp1,sp3) \tkzGetPoint{y10}
\tkzDefMidPoint(sp3,mz1) \tkzGetPoint{y11}
\tkzDefMidPoint(sp3,mp) \tkzGetPoint{y12}
\tkzDefMidPoint(sp3,msp2) \tkzGetPoint{y13}
\tkzDefMidPoint(mp,msp2) \tkzGetPoint{y14}
\tkzDefMidPoint(z2,msp2) \tkzGetPoint{y15}
\tkzDefMidPoint(msp2,sp1)\tkzGetPoint{y17}
\tkzDefMidPoint(z2,sp1)\tkzGetPoint{y18}
\tkzDrawPolygon[dgreen,thick](y1,y2,y3,y4,y5,y6,y7,y71,y72,y8,y9,y10,y11,y12,y13,y14,y15,y17,y18)
\tkzLabelPoint[left](p){$z^{3,2}$}
\tkzLabelPoint[right](msp3){$-s\gamma_3p^3$}
\tkzLabelPoint[left](sp3){$s\gamma_3 p^3$}
\tkzLabelPoint[right](sp2){$s\gamma_2p^2$}
\tkzLabelPoint[left](msp2){$-s\gamma_2p^2$}
\tkzLabelPoint[left](sp1){$s\gamma_1 p^1$}
\tkzLabelPoint[right](msp1){$-s\gamma_1 p^1$}
\tkzLabelPoint[right](zero){$0$}
\tkzLabelPoint[below left](z1){$z^{3,1}$}
\tkzLabelPoint[above left](mz2){$z^{1,2}$}
\tkzLabelPoint[below right](z2){$z^{2,1}$}
\tkzLabelPoint[above right](mz1){$z^{1,3}$}
\tkzLabelPoint[right](mp){$z^{2,3}$}
\tkzDrawPoints[fill=black](zero,sp1,msp3,sp2,z1,z2,p,mz2,mz1,msp1,sp3,msp2,mp)
\end{tikzpicture}}
\caption{Construction used in the proof of  Part b) Step 1 of \Cref{lem:beforePlanarCase}: we replace $-sp^i$ by $-s\gamma_i p^i$ for $i=1,2,3$ and obtain $K'_1$ (black), $-K'_1$ (grey) and $\frac{K'_1-K_1'}{2}$ (green). 
}
    \label{fig:step1}
\end{figure}

On the one hand, observe for $i=1,2,3$ that 
\begin{align*}
\frac{1}{s}K'_1 \cap \pos(\{ z^{i,i+1},z^{i,i+2} \}) &= \conv\left(\left\{0, -\gamma_i p^i, \frac{1}{s} z^{i,i+1}, \frac{1}{s} z^{i,i+2} \right\}\right) \\ &\subset \conv(\{0, z^{i,i+1},z^{i,i+2} \})\\ &= -K'_1 \cap \pos(\{ z^{i,i+1},z^{i,i+2} \}), 
\end{align*}
and 
\begin{align*}
\frac{1}{s}K'_1 \cap \pos(\{ -z^{i,i+1},-z^{i,i+2} \}) &= 
\conv\left(\left\{0, -\frac{1}{s} z^{i,i+1}, -\frac{1}{s} z^{i,i+2} \right\}\right) \\ &\subset \conv(\{0, -z^{i,i+1},-z^{i,i+2} \}) \\ &\subset 
-K'_1 \cap \pos(\{ -z^{i,i+1},-z^{i,i+2} \}),  
\end{align*}
which implies $\frac{1}{s}K'_1 \subset -K'_1$. Since also $\gamma_i p^i, -s \gamma_i p^i \in \bd(K'_1)$, $i=1,2,3$, the set $\set{\gamma_1p^1,\gamma_2p^2,\gamma_3p^3}$ forms a triple of well spread asymmetry points and we see that $K'_1$ is Minkowski centered with $s(K'_1)=s(K')=s$. 

The direction in which $\tau(K'_1$) is attained might have changed compared to $K'$ but since $K'_1$ has a similar structure as $K'$, by Part a) there exists $p \in \bd(K'_1 \cap (-K'_1)) \cap \tau(K'_1) \bd \left(\frac{K'_1-K'_1}{2}\right)$ with $p \in \bd(K'_1) \cap \bd(-K'_1)$. 
Observe that since
\[
\rho_1(z^{3,2}) z^{3,2} \in \left[\frac{-\gamma_3 sp^3+z^{3,2}}{2}, \frac{-\gamma_3 sp^3+\gamma_2 sp^2}{2}\right] \cup \left[\frac{\gamma_2sp^2+z^{3,2}}{2}, \frac{-\gamma_3sp^3+\gamma_2sp^2}{2}\right], 
\]
we have $\rho_1(z^{3,2}) \leq \rho(z^{3,2})$. Similarly, one can show that 
for all $i=1,2,3$ and $j=i+1,i+2$ we have $\rho_1(z^{i,j})\leq \rho(z^{i,j})$ and therefore, 
\[
\tau(K'_1) \geq \tau(K'). 
\] 

Let us assume w.l.o.g. from now on that 
$p=z^{3,2}$. 
\vspace{3mm}
\begin{enumerate}[Step 2:]
\item Let $\hat p^2, \hat p^3$ be defined such that $- \hat p^3 \in [ z^{3,1} , p]$, $ \hat p^2 \in [z^{1,2} , p]$ and $- s\hat p^2, - s\hat p^3 
\in \aff(\{ z^{2,1},z^{3,1}\})$. This means that we move the points $-s\gamma_3p^3, s\gamma_2p^2$ parallel to the segments $[ z^{3,1} , p]$ and $[z^{1,2} , p]$.

Replace $K'_1$ by 
\[
K'_2=\conv\left(\{-s \gamma_1 p^1, -s \hat p^2, -s\hat p^3, z^{i,i+1},z^{i,i+2}: i=1,2,3\}\right),
\]
see \Cref{fig:step2}. 

We show that for $K'_2$ holds 
\begin{enumerate}[(i)]
  \item $s(K'_2)=s(K'_1)$, 
  \item $\tau(K'_2) \geq \tau(K'_1)$, and
   \item $p \in \bd(K'_2 \cap (-K'_2)) \cap \tau(K'_2) \bd \left(\frac{K'_2-K'_2}{2}\right)$. 
\end{enumerate}
\end{enumerate}
Since $-\hat p^3 \in [z^{3,1}, z^{3,2}]$, 
\begin{align*}
\frac{1}{s}K'_2 \cap \pos(\{ z^{3,1},z^{3,2} \}) &= \conv\left(\left\{0, -\hat p^3, \frac{1}{s} z^{3,1}, \frac{1}{s} z^{3,2} \right\}\right) \\ &\subset \conv(\{0, z^{3,1},z^{3,2} \})\\ &= -K'_2 \cap \pos(\{ z^{3,1},z^{3,2} \}), 
\end{align*}
and 
\[
\frac{1}{s}K'_2 \cap \pos(\{ z^{1,2},z^{3,2} \}) \subset -K'_2 \cap \pos(\{ z^{1,2},z^{3,2} \}). 
\]

Thus, we conclude $-\frac{1}{s}K'_2 \subset K'_2$. 
The points $ \gamma_1 p^1, \hat p^2, \hat p^3$ form a well-spread triple of asymmetry points of $K'_2$ by \Cref{lem:sixareas}, which implies that $s(K'_2)=s$. 

As shown in Part a), we only need to consider the points in $\bd(K'_2) \cap \bd(-K'_2)$. 
We know that
\begin{equation*}
    \rho_2(z^{3,1}) z^{3,1} \in \left[\frac{-s\hat p^3+z^{3,1}}{2}, \frac{-s\hat p^3+ s\gamma p^1}{2}\right] \cup \left[\frac{s\gamma_1 p^1+z^{3,1}}{2}, \frac{-s\hat p^3+s\gamma_1 p^1}{2}\right]
\end{equation*}
and 
\begin{equation*}
    \rho_1(z^{3,1}) z^{3,1} \in \left[\frac{-s\gamma_3p^3+z^{3,1}}{2}, \frac{-s\gamma_3p^3+ s\gamma_1 p^1}{2}\right] \cup \left[\frac{s\gamma_1 p^1+z^{3,1}}{2}, \frac{-s\gamma_3p^3+s\gamma_1 p^1}{2}\right]. 
\end{equation*}

Let $H^=_2= \aff\left(\set{\frac{s\gamma_1 p^1+z^{3,1}}{2}, \frac{-s\hat p^3+s\gamma_1 p^1}{2}}\right)$, 
$H_3^= = \aff\left(\set{\frac{-s\hat p^3+z^{3,1}}{2}, \frac{-s\hat p^3+s\gamma_1 p^1}{2}}\right)$ and $H_4^= = \aff\left(\set{\frac{-s \gamma_3p^3+z^{3,1}}{2}, \frac{-s\gamma_3 p^3+s\gamma_1 p^1}{2}}\right)$ (see \Cref{fig:step2}). Let $H_k^{\leq}$, $k=2,3,4$ denote the halfspace 
containing the origin.

Since $H^=_2= \frac{s \gamma_1+1}{2} \aff\left(z^{3,1},-s\hat p^3\right)$ and the hyperplanes $H^{=}_3$, $H^{=}_4$ are parallel to $\left[z^{3,1},s \gamma_1 p^1\right]$, 
\[
\frac{-s\gamma_3p^3+z^{3,1}}{2}, \frac{-s\hat p^3+ s\gamma_1 p^1}{2}, \frac{s\gamma_1 p^1+z^{3,1}}{2} \in H^\leq_2 \cap H^\leq_3,   
\] 
which implies
\[
\rho_2(z^{3,1}) \geq \rho_1(z^{3,1}). 
\]



Analogously, we obtain $\rho_2(z^{1,2}) \geq \rho_1(z^{1,2})$. 

Since $\rho_2(z^{2,1})=\rho_2(z^{1,2})$,$\rho_1(z^{2,1})=\rho_1(z^{1,2})$, $\rho_2(z^{1,3})=\rho_2(z^{3,1})$, and $\rho_1(z^{1,3})=\rho_1(z^{3,1})$,  we have 
\[
\rho_2(z^{2,1}) \geq \rho_1(z^{2,1}) \quad \text{and} \quad \rho_2(z^{1,3}) \geq \rho_1(z^{1,3}). 
\] 

Next, we know
\begin{equation*}
    \rho_2(p) p \in \left[\frac{-s\hat p^3+p}{2}, \frac{-s\hat p^3+ s\hat p^2}{2}\right] \cup \left[\frac{s\hat p^2+p}{2}, \frac{-s\hat p^3+s\hat p^2}{2}\right].
\end{equation*}


Similarly to before, we see that 
   $ \left[\frac{-s\hat p^3+p}{2}, \frac{-s\hat p^3+ s\hat p^2}{2}\right] \cup \left[\frac{s\hat p^2+p}{2}, \frac{-s\hat p^3+s \hat p^2}{2}\right]$ is contained in the half-spaces containing the origin defined by $\aff\left( \set{\frac{-s \gamma_3p^3+p}{2}, \frac{-s\gamma_3 p^3+ s\gamma_2p^2}{2}}\right)$ and $\aff\left( \set{\frac{s\gamma_2p^2+p}{2}, \frac{-s \gamma_3p^3+ s\gamma_2p^2}{2}}\right)$.

Hence, one can show that $\rho_2(\pm p)\leq \rho_1(\pm p)$ and 
\[
\rho_2^{-1}(p)=\tau(K'_2) \geq \tau(K'_1). 
\] 

\begin{figure}[ht]
\centering
\scalebox{0.95}{
 \begin{tikzpicture}[scale=8]
\tkzDefPoints{0/0/zero, 0/.66/sp1, .7/0.3/msp3, 0.4/0.4/z1,-0.4/-0.4/mz1,  -0.35/0.4/z2, 0.35/-0.4/mz2,  0.6/0.066/p, 0.7/-.33/sp2}
\tkzDrawPoints[fill=black](zero,sp1,msp3,sp2,z1,z2,p,mz2)
\tkzDrawSegments[thick, gray](sp1,z1 z1,p p,sp2 sp2,mz2 sp1,z2)
\tkzDrawSegments[thick](z2,z1 z1,msp3 p,msp3 p,mz2)
\tkzDefLine[parallel=through msp3](p,z1)  \tkzGetPoint{x1}
\tkzInterLL(z1,z2)(msp3,x1) \tkzGetPoint{msp3h}
\tkzDrawPoint[fill=black](msp3h)

\tkzDefLine[parallel=through sp2](p,mz2)  \tkzGetPoint{x2}
\tkzInterLL(mz1,mz2)(sp2,x2) \tkzGetPoint{sp2h}
\tkzDrawSegments[dotted](sp2h,sp2 msp3h,msp3)
\tkzDrawSegments[dashed](sp2h,p p,msp3h mz2,sp2h z1,msp3h)
\tkzDrawPoint[fill=black](sp2h)
\tkzDefMidPoint(sp1,z1) \tkzGetPoint{y1}
\tkzDefMidPoint(sp1,msp3) \tkzGetPoint{y2}
\tkzDefMidPoint(msp3,z1) \tkzGetPoint{y3}
\tkzDefMidPoint(msp3,p) \tkzGetPoint{y4}
\tkzDefMidPoint(msp3,sp2) \tkzGetPoint{y5}
\tkzDefMidPoint(p,sp2) \tkzGetPoint{y6}
\tkzDrawSegments[dgreen, thick](y1,y2 y2,y3 y3,y4 y4,y5 y5,y6)
\tkzDefMidPoint(sp1,msp3h) \tkzGetPoint{q2}
\tkzDefMidPoint(msp3h,z1) \tkzGetPoint{q3}
\tkzDefMidPoint(msp3h,p) \tkzGetPoint{q4}
\tkzDefMidPoint(msp3h,sp2h) \tkzGetPoint{q5}
\tkzDefMidPoint(p,sp2h) \tkzGetPoint{q6}
\tkzDrawSegments[orange, thick](y1,q2 q2,q3 q3,q4 q4,q5 q5,q6)
\tkzDrawSegment[dgreen, dotted](y6,q6)
\tkzDrawPoints[orange](y1,q2,q3,q4,q5,q6)
\tkzDrawPoints[dgreen](y2,y3,y4,y5,y6)
\tkzLabelPoint[left](p){$p$}
\tkzLabelPoint[right](msp3){$-s\gamma_3p^3$}
\tkzLabelPoint[right](msp3h){$-s\hat p^3$}
\tkzLabelPoint[right](sp2){$s\gamma_2p^2$}
\tkzLabelPoint[right](sp2h){$s\hat p^2$}
\tkzLabelPoint[left](sp1){$s\gamma_1 p^1$}
\tkzLabelPoint[right](zero){$0$}
\tkzLabelPoint[below left](z1){$z^{3,1}$}
\tkzLabelPoint[above left](mz2){$z^{1,2}$}
\tkzLabelPoint[below right](z2){$z^{2,1}$}
\tkzDrawLine[add= 1 and 1, dashed, orange ](y1,q2)
\tkzLabelLine[above left, orange, pos=-.65](y1,q2){$H^=_2$}
\tkzDrawLine[add= 1 and 0.2, dashed, orange ](q2,q3)
 \tkzLabelLine[above, orange, pos=-.7](q2,q3){$H^=_3$}
 \tkzDrawLine[add= .6 and 0.4, dashed, dgreen ](y2,y3) 
 \tkzLabelLine[above left, dgreen, pos=-.5](y2,y3){$H^=_4$}
 
 \tkzDrawLine[add= .4 and .5, dashed, dgreen](y4,y5)
\tkzDrawLine[add= .5 and 1, dashed, dgreen ](y5,y6)

\tkzLabelPoint[below left, orange](y1){$\frac{s\gamma_1 p^1+z^{3,1}}{2}$}
\tkzLabelPoint[right, dgreen](y2){$\quad \frac{s\gamma_1 p^1-s\gamma_3p^3}{2}$}
\tkzLabelPoint[right, dgreen](y4){$\frac{p-s\gamma_3p^3}{2}$}
\tkzLabelPoint[right, dgreen](y5){$\frac{s\gamma_2p^2-s\gamma_3p^3}{2}$}
\tkzLabelPoint[right,dgreen](y6){$\frac{s\gamma_2p^2+p}{2}$}
\tkzLabelPoint[above right, orange](q2){$ \frac{s\gamma_1 p^1-s\hat p^3}{2}$}
\tkzLabelPoint[left,orange](q6){$\frac{s\hat p^2+p}{2}$}

\tkzDrawPoints[fill=black](zero,sp1,msp3,sp2,z1,z2,p,mz2)

\end{tikzpicture}}
\caption{Construction used in the proof of  Part b) Step 2 of \Cref{lem:beforePlanarCase}: we replace $-s\gamma_3p^3$ by $-s\hat p^3$ and $s\gamma_2p^2$ by $s\hat p^2$. Parts of the boundaries of $\frac{K_1-K_1}{2}$ (green) and $\frac{K_2-K_2}{2}$ (orange) are shown.}
    \label{fig:step2}
\end{figure}
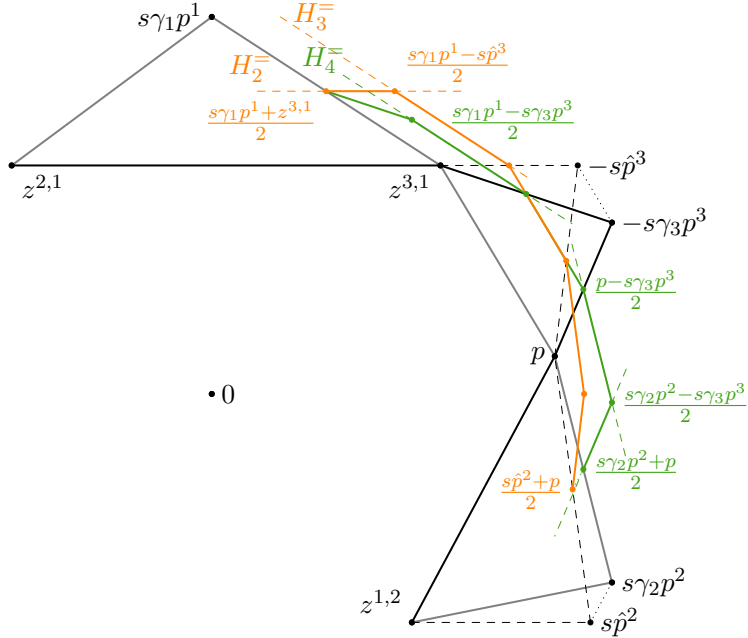

For better readability, let us denote $\gamma_1 p^1$, $\hat p^2$ and $\hat p^3$ by $p^1$, $p^2$ and $p^3$, respectively.

\vspace{3mm}
\begin{enumerate}[Step 3:]
\item Replace $K'_2$ by 
\[
K'_3=\conv\left(\{-s p^1, -s p^2, -s p^3, p^2, p^3, \pm p \}\right),
\] 
see \Cref{fig:step3-1}. 
We show that for $K'_3$ holds 
\begin{enumerate}[(i)]
  \item $s(K'_3)=s(K'_2)$, 
  \item $\tau(K'_3) \geq  \tau(K'_2)$, and
  \item $p \in (\bd(K'_3) \cap \bd(-K'_3)) \cap \tau(K'_3) \bd \left(\frac{K'_3-K'_3}{2}\right)$.
\end{enumerate}
\end{enumerate}
\vspace{3mm}
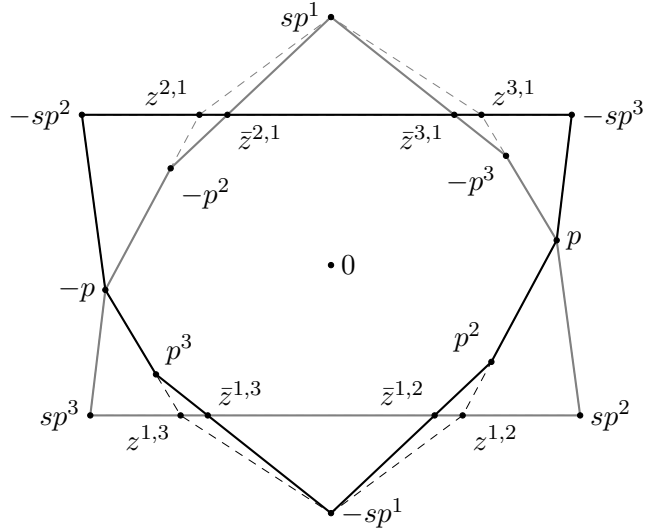
\begin{figure}[ht]
\centering
\scalebox{1}{
 \begin{tikzpicture}[scale=5]
\tkzDefPoints{0/0/zero, 0/.66/sp1, .7/0.3/msp3, 0.4/0.4/z1,-0.4/-0.4/mz1,  -0.35/0.4/z2, 0.35/-0.4/mz2,  0.6/0.066/p, 0.7/-.33/sp2, 0/-0.66/msp1, -0.7/-0.3/sp3, -0.6/-0.066/mp, -0.7/.33/msp2}

\tkzDefLine[parallel=through msp3](p,z1)  \tkzGetPoint{x1}
\tkzInterLL(z1,z2)(msp3,x1) \tkzGetPoint{msp3h}

\tkzDefLine[parallel=through sp2](p,mz2)  \tkzGetPoint{x2}
\tkzInterLL(mz1,mz2)(sp2,x2) \tkzGetPoint{sp2h}
\tkzDefPointBy[symmetry= center zero](msp3h)
 \tkzGetPoint{sp3h}

 \tkzDefPointBy[symmetry= center zero](sp2h) \tkzGetPoint{msp2h}

\tkzInterLL(zero,msp3h)(z1,p) \tkzGetPoint{mp3h}
\tkzInterLL(zero,msp2h)(z2,mp) \tkzGetPoint{mp2h}

\tkzInterLL(zero,sp3h)(mz1,mp) \tkzGetPoint{p3h}
\tkzInterLL(zero,sp2h)(mz2,p) \tkzGetPoint{p2h}

 \tkzInterLL(z1,z2)(mp2h,sp1) \tkzGetPoint{z2h}
  \tkzInterLL(z1,z2)(mp3h,sp1) \tkzGetPoint{z1h}
  \tkzInterLL(mz1,mz2)(p3h,msp1) \tkzGetPoint{mz1h}
  \tkzInterLL(mz1,mz2)(p2h,msp1) \tkzGetPoint{mz2h}

\tkzDrawPolygon[ dashed](z2,z1,msp3h,p,mz2,msp1,mz1,mp,msp2h)
\tkzDrawPolygon[gray, dashed](mz2,mz1,sp3h,mp,z2,sp1,z1,p,sp2h)
\tkzDrawPolygon[thick](msp3h,p,p2h,msp1,p3h,mp,msp2h)
\tkzDrawPolygon[thick, gray](sp3h,mp,mp2h,sp1,mp3h,p,sp2h)
\tkzDrawPoints[fill=black](msp3h,sp2h,msp2h,sp3h, mp3h, mp2h,p3h,p2h)

\tkzLabelPoint[right](p){$p$}
\tkzLabelPoint[right](msp3h){$-sp^3$}
\tkzLabelPoint[left](sp3h){$s p^3$}
\tkzLabelPoint[right](sp2h){$sp^2$}
\tkzLabelPoint[left](msp2h){$-sp^2$}
\tkzLabelPoint[left](sp1){$s p^1$}
\tkzLabelPoint[right](msp1){$-sp^1$}
\tkzLabelPoint[right](zero){$0$}
\tkzLabelPoint[above right](z1){$z^{3,1}$}
\tkzLabelPoint[below right](mz2){$z^{1,2}$}
\tkzLabelPoint[above left](z2){$z^{2,1}$}
\tkzLabelPoint[below left](mz1){$z^{1,3}$}

\tkzLabelPoint[below left](z1h){$\bar z^{3,1}$}
\tkzLabelPoint[above left](mz2h){$\bar z^{1,2}$}
\tkzLabelPoint[below right](z2h){$ \bar z^{2,1}$}
\tkzLabelPoint[above right](mz1h){$\bar z^{1,3}$}
\tkzLabelPoint[left](mp){$-p$}
\tkzLabelPoint[below left](mp3h){$-p^3$}
\tkzLabelPoint[below right](mp2h){$-p^2$}
\tkzLabelPoint[above right](p3h){$p^3$}
\tkzLabelPoint[above left](p2h){$p^2$}

\tkzDrawPoints[fill=black](zero,sp1,z1,z2,p,mz2,mz1,msp1,mp, z1h,z2h,mz1h,mz2h)
\end{tikzpicture}}
\caption{Construction used in the proof of  Part b) Step 3 of \Cref{lem:beforePlanarCase}: We obtain $K'_3$ (black) and $-K'_3$ (grey).}
    \label{fig:step3-1}
\end{figure}

Observe that 
\begin{align*}
\frac{1}{s}K'_3 \cap \pos(\{ -sp^2, -sp^3 \}) &= \conv\left(\left\{0, - p^2, -p^3 \right\}\right) \\ &\subset \conv(\{0, - p^2, -p^3, sp^1 \})\\ &= -K'_3 \cap \pos(\{  -sp^2, -sp^3 \}), 
\end{align*}
and
\begin{align*}
\frac{1}{s}K'_3 \cap \pos(\{ sp^2, -sp^3 \}) &= \frac{1}{s}K'_2 \cap \pos(\{ sp^2, -sp^3 \}) \}), \\
-K'_3 \cap \pos(\{ sp^2, -sp^3 \}) &= -K'_2 \cap \pos(\{ sp^2, -sp^3 \}). 
\end{align*}
Thus, we conclude $-\frac{1}{s}K'_3 \subset K'_3$.

The points $ p^1, p^2, p^3$ also form a well-spread triple of asymmetry points of $K'_3$, which implies $s(K'_3)=s$. 

We denote $\bar z^{i,j} \in (\bd(K'_3) \cap \bd(-K'_3))$, $i,j=1,2,3$. 

Then by Part a) there exists $\bar p \in (\bd(K'_3) \cap \bd(-K'_3)) \cap \tau(K'_3) \bd \left(\frac{K'_3-K'_3}{2}\right)$.

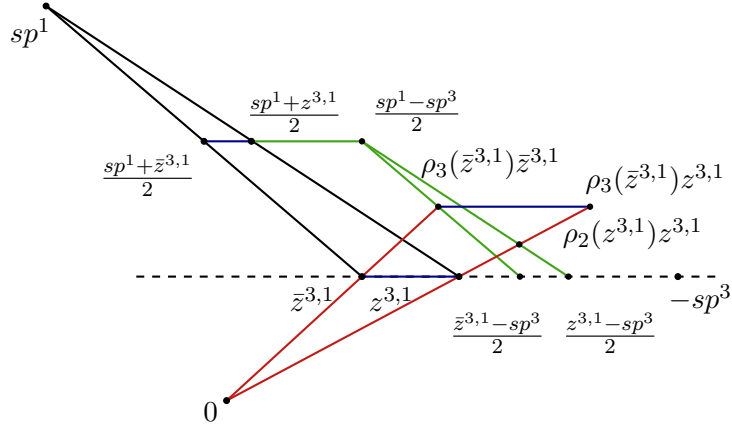
\begin{figure}[H]
\centering
 \begin{tikzpicture}[scale=3]
\draw [thick, dashed] (-1,0)  -- (1.6,0);
\draw [dblue, thick] (-0.49,0.6)  -- (-0.7,0.6);
\draw [dblue, thick] (0,0)  -- (0.43,0);
\draw [thick] (0.43,0)  -- (-1.4,1.2);
\draw [thick] (-1.4,1.2)  -- (0,0);

\tkzDefPoints{-0.6/-0.55/zero, -1.4/1.2/sp1, 1.4/0/msp3, 0.43/0/z, 0/0/zb}
\tkzDefMidPoint(z,sp1) \tkzGetPoint{x1}
\tkzDefMidPoint(zb,sp1) \tkzGetPoint{y1}
\tkzDefMidPoint(z,msp3) \tkzGetPoint{x2}
\tkzDefMidPoint(zb,msp3) \tkzGetPoint{y2}
\tkzDefMidPoint(msp3,sp1) \tkzGetPoint{x3}

\tkzDrawSegments[dgreen,thick](x1,x3 x3,x2 x3,y2)
\tkzInterLL(zero,z)(x3,x2)\tkzGetPoint{rho2}
\tkzInterLL(zero,zb)(x3,y2)\tkzGetPoint{rho3}
\tkzDefLine[parallel=through rho3](z,zb) 
\tkzGetPoint{r}
\tkzInterLL(rho3,r)(zero,z) \tkzGetPoint{rho3z}

\tkzDrawSegment[dblue,thick](rho3z,rho3)
\tkzDrawSegments[dred, thick](zero,rho3z zero,rho3)
\tkzDrawPoints[black,size=2](x2,y2,x3)
\tkzDrawPoints[fill=black,size=2](rho2,rho3,rho3z)
\tkzLabelPoint[above right](x3){$\frac{sp^1-sp^3}{2}$}
\draw (1.1,-0.25) node {$\frac{z^{3,1}-sp^3}{2}$};
\draw (0.6,-0.25) node {$\frac{\bar z^{3,1}-sp^3}{2}$};

\draw [fill] (0,0) circle [radius=0.013]; 
\draw [fill] (-0.6,-0.55) circle [radius=0.013]; 
\draw [fill] (0.43,0) circle [radius=0.013]; 
\draw [fill] (1.4,0) circle [radius=0.013]; 
\draw [fill] (-1.4,1.2) circle [radius=0.013]; 
\draw [fill] (-0.7,0.6) circle [radius=0.013]; 
\draw [fill] (-0.49,0.6) circle [radius=0.013]; 

\draw (-0.67,-0.6) node {$0$};
\draw (-0.22,-0.1) node {$\bar z^{3,1}$};
\draw (0.13,-0.1) node {$z^{3,1}$};
\draw (1.5,-0.1) node {$-sp^3$};
\draw (1.3,0.43) node {$\rho_3(\bar z^{3,1}) z^{3,1}$};
\draw (0.57,0.5) node {$\rho_3(\bar z^{3,1}) \bar z^{3,1}$};
\draw (1.2,0.19) node {$\rho_2(z^{3,1}) z^{3,1}$};
\draw (-1.47,1.08) node {$sp^1$};
\draw (-0.95,0.45) node {$\frac{sp^1+\bar z^{3,1}}{2}$};
\draw (-0.3,0.73) node {$\frac{sp^1+z^{3,1}}{2}$};

\end{tikzpicture}
\caption{Construction used in the proof of  Part b) Step 3 of \Cref{lem:beforePlanarCase}. We see that $\rho_3(\bar z^{3,1})\geq \rho_2(z^{3,1}) $. 
}
    \label{fig:step3}
\end{figure}


Next we show
\[
\rho_3(\bar z^{i,j}) \geq \rho_2(z^{i,j}), \quad \text{ for } i,j=1,2,3 \quad \text{ such that } z^{i,j}\neq \pm p. 
\] 
Observe that it suffices to show 
\[
\rho_3(\bar z^{3,1}) \geq \rho_2(z^{3,1}),
\] 
the argument for $z^{2,1}$ follows analogously. 

Indeed, $\bar z^{3,1} \in [-sp^2, -sp^3]$, $\aff(\{ \frac{sp^1+\bar z^{3,1}}{2}  , \frac{sp^1+z^{3,1}}{2}\})$ is parallel to 
$\aff(\{ \bar z^{3,1} , z^{3,1}\})$, and thus, also parallel to $\aff(\{ \rho_3(\bar z^{3,1}) \bar z^{3,1}, \rho_3(\bar z^{3,1}) z^{3,1}\})$ (see \Cref{fig:step3}). The distance between $\frac{z^{3,1}-sp^3}{2}$ and $\frac{\bar z^{3,1}-sp^3}{2}$ is smaller than the distance between $\rho_3(\bar z^{3,1}) z^{3,1}$ and $\rho_3(\bar z^{3,1}) \bar z^{3,1}$.  This implies $\rho_3(\bar z^{3,1}) \geq \rho_2(z^{3,1})$, see \Cref{fig:step3}.

Furthermore,  $p \in (\bd(K'_3) \cap \bd(-K'_3)) \cap \tau(K'_3) \bd \left(\frac{K'_3-K'_3}{2}\right)$, thus, 
\[
\tau(K'_3) = \tau(K'_2). 
\] 

\vspace{3mm}
\begin{enumerate}[Step 4:]
\item  We define $d^1:=\aff (\{-p,p^3\}) \cap \aff (\{p,p^2\})$ and 
$p_{*}^1 \in \aff (\{-sp^2,-sp^3\})$,
such that $sp_{*}^1 = \gamma (-d^1)$ for some $\gamma \leq \frac{1}{s}$. Replace $K'_3$ by 
\[
K'_4=\conv\left(\{-s p_{*}^1, -s p^2, -s p^3, p^2, p^3, \pm p \}\right). 
\]

We show that for $K'_4$ holds 
\begin{enumerate}[(i)]
  \item $s(K'_4)=s(K'_3)$, 
  \item $\tau(K'_4) =  \tau(K'_3)$, and
  \item $p \in \bd((K'_4) \cap (-K'_4)) \cap \tau(K'_4) \bd \left(\frac{K'_4-K'_4}{2}\right)$.
\end{enumerate}
\end{enumerate}
\vspace{3mm}
As before, note that $-\frac{1}{s}K'_4 \subset K'_4$ and $p_{*}^1, p^2, p^3$ form a well-spread triple of asymmetry points of $K'_4$, which implies that $s(K'_4)=s$.

Next, we show 
\[
\rho_3(\bar  z^{1,2}) \bar  z^{1,2} = \frac{sp^2-sp^1}{2}
\] 
for the $\bar  z^{i,j}$ from the previous step. 
Since $\bar z^{1,2} \in \pos (\{-sp^1,sp^2\})$, there exist $\alpha, \beta \geq 0$, such that 
\[
\bar  z^{1,2} = \alpha (-sp^1) + \beta (sp^2)=   \alpha (-sp^1) + s \beta (p^2). 
\] 
Thus, 
\[
\frac{\|\bar  z^{1,2}- p^2\|}{\|-sp^1-\bar  z^{1,2}\|}=\frac{\alpha}{s\beta} . 
\] 


Since $p^2= \lin (\{sp^2\}) \cap \aff (\{ p^1,\ \bar  z^{1,2}\})$, we have 
\[
\frac{\|-sp^1- p^2\|}{\|-sp^1-\bar  z^{1,2}\|} = 1+\frac{\|\bar  z^{1,2}- p^2\|}{\|-sp^1-\bar  z^{1,2}\|}=\frac{s-\frac1s }{s-1}, 
\] 
i.e., 
\[
\frac{\|\bar  z^{1,2}- p^2\|}{\|-sp^1-\bar  z^{1,2}\|}=\frac{1}{s}. 
\]  
Combining those two facts, we obtain 
\[
\frac{\alpha}{s\beta}= \frac{1}{s}, 
\] 
i.e., $\alpha= \beta$. This implies 
\[
\rho_3(\bar  z^{1,2}) \bar  z^{1,2} = \frac{sp^2-sp^1}{2}. 
\] 
We denote by $z_{*}^{i,j}$ the points in $\bd(K'_4)\cap \bd(K'_4)$. Since $K'_4$ has a similar structure as $K'_3$, we have 
\[
\rho_4(z_{*}^{1,2}) \bar  z^{1,2} = \frac{sp^2-sp_{*}^1}{2}. 
\] 
From the fact that $\aff(\{-sp^1,-sp_{*}^1\})$ is parallel to $\aff(\{\bar  z^{1,2},z_{*}^{1,2}, sp^2\})$ follows 
\[ 
\rho_4( z_{*}^{1,2}) = \rho_3( \bar  z^{1,2}).
\] 
Similarly, one can show 
\[
\rho_3(\bar  z^{1,3}) \bar  z^{1,3} = \frac{sp^3-sp^1}{2}  
\] 
and 
\[
\rho_4( z_{*}^{1,3}) = \rho_3( \bar  z^{1,3}).
\] 

Since $\frac{K'_4-K'_4}{2} \cap \pos (\{-sp^3,sp^2\}) = \frac{K'_3-K'_3}{2} \cap \pos (\{-sp^3,sp^2\})$, we have
\[
\rho_4( p) = \rho_3( p) = \tau^{-1} ( K'_3).
\]  

Thus, we have shown that for any $v \in \ext((K'_4) \cap (-K'_4))$ holds
\[
\rho_4( v) = \rho_3( v). 
\]
Remembering that $\rho_3( p) = \tau^{-1} ( K'_3)$, we finally obtain $\tau ( K'_3)=\tau ( K'_4)$ and 
\[
\rho_4( p) = \tau^{-1} ( K'_4).
\] 

Since $p_{*}^1 = \gamma (-d^1)$, we have  
\[
\bar K=K'_4=\conv(\{\pm p,p^2,p^3,\gamma s d^1,-s p^2, -s p^3\}).  
\]  
\end{enumerate}
\end{proof}

\begin{lemma}\label{lem:PlanarCase}
Let $K\in\mathcal K^2$ be Minkowski centered such that $s(K)>\varphi$. Then 
\begin{equation}\label{eq:bound_alpha}
    \tau(K)\leq c(s(K)).
\end{equation}
\end{lemma}

\begin{proof}
Considering $\bar K$ as in \Cref{lem:beforePlanarCase}, we have $\tau(K) \leq \tau(\bar K)$ and $s(K)=s(\bar K)$. Hence, if $\bar K$ fulfills \eqref{eq:bound_alpha}, we have 
\[
\tau(K) \leq \tau(\bar K) \leq c(s(\bar K))= c(s(K)),  
\]
i.e., $K$ also fulfills \eqref{eq:bound_alpha}. Thus, it suffices to show \eqref{eq:bound_alpha} for $\bar K$. 
Since by \cite[Theorem 1.7]{BDG1} $\frac{2}{s(K)+1} \leq \tau(K) \leq 1$ and $s \in \left[1,2\right]$, we have $\tau \in \left[\frac23,1\right]$.

Now, for any fixed $\tau \in \left[\frac23,1\right]$ we determine the maximal $s=s(\tau)$, such that there exists a Minkowski centered $\bar K \in \K^2$ with $\tau(\bar K) = \tau$ and $s(\bar K)=s$. 

\begin{figure}[ht]
    \centering
 	\begin{tikzpicture}[scale=3]
	\draw [thick, dashed] (-1,0)  -- (1,0); 
	\draw [thick] (-1, -2) -- (-1,1.5); 
	\draw [thick] (1, -2) -- (1,1.5); 
	\draw [thick] (0.8, -2) -- (0.8,1.5); 
	\draw [thick] (-0.8, -2) -- (-0.8,1.5); 

 \draw[->] [thick] (1, 1.3) -- (1.2, 1.3) node[right] {$u$};
	\draw [thick] (-1,0.7) -- (1,0.93); 
	\draw [thick, dashed, gray] (0.03,0.82) -- (-0.08,-1.9); 
	\draw [thick, dashed, gray] (-0.59,-0.53) -- (1,0.93); 
	\draw [thick, dashed, gray] (0.61,-0.42) -- (-1,0.7); 
	\draw [thick]  (-1,0.7)--(1,0.93); 
    \draw [fill=lgold, fill opacity=0.7]  (-1,0.7)--(-0.8,0)--(-0.59,-0.54)--(-0.05,-1.3)--(0.6,-0.42)--(0.8,0)--(1,0.93)--(-1,0.7);

 \draw [thick, dred]  (-1,0.7)--(-0.8,0)--(-0.08,-1.9)--(0.8,0)--(1,0.93)--(-1,0.7);
	
	
	\draw [fill] (0,0) circle [radius=0.02];
	\draw [fill] (-1,0) circle [radius=0.02]; 
	\draw [fill] (1,0) circle [radius=0.01]; 
	\draw [fill] (-0.8,0) circle [radius=0.02]; 
	\draw [fill] (0.8,0) circle [radius=0.02]; 
	\draw [fill] (1,0.93) circle [radius=0.02]; 
	\draw [fill] (-1,0.7) circle [radius=0.02]; 
	\draw [fill] (0.03,0.82) circle [radius=0.02]; 
	\draw [fill] (-0.59,-0.54) circle [radius=0.02]; 
	\draw [fill] (0.6,-0.42) circle [radius=0.02]; 
	\draw [fill] (-0.08,-1.9) circle [radius=0.02]; 
	\draw [fill] (-0.05,-1.3) circle [radius=0.02]; 
    \tkzDefPoint(0,0){zero}
	\draw (-0.16,-0.07) node {$0$};
	\draw (0.7,0.07) node {$p$};
	\draw (-0.7,0.07) node {$-p$};
	\draw (1.15,0.07) node {$\frac{1}{\alpha}p$};
	\draw (-1.15,0.07) node {$-\frac{1}{\alpha}p$};
	\draw (0.7,-0.4) node {$p^2$};
	\draw (-0.72,-0.5) node {$p^3$};
 \tkzDefPoint(-0.05,-1.3){msp1}
    \draw (0.12,-1.2) node {$-sp^1$};
	\draw (0.12,-1.9) node {$d^1$};
	\draw (0.1,0.95) node {$p^1$};
 \tkzDefPoint(-1,0.7){msp2}
 \tkzDefPoint(1,0.93){msp3}
	\draw (-1.17,0.7) node {$-sp^2$};
	\draw (1.15,0.93)  node {$-sp^3$};
\tkzDefPointBy[symmetry=center zero](msp1)
\tkzGetPoint{sp1}
\tkzDefPointBy[symmetry=center zero](msp2)
\tkzGetPoint{sp2}
\tkzDefPointBy[symmetry=center zero](msp3)
\tkzGetPoint{sp3}
\tkzDrawPoints[fill=black,size=3.5](sp1,sp2,sp3)
\tkzDefPoint(0.6,-0.42){p2}
\tkzDefPoint(-0.59,-0.55){p3}
\tkzDefPoint(-0.6,0.42){mp2}
\tkzDefPoint(0.59,0.55){mp3}
\tkzDefPoint(0.8,0){p}
\tkzDefPoint(-0.8,0){mp}
\tkzDrawPoints[fill=black,size=3.5](mp2,mp3)
\tkzDrawPolygon[dashed, gray](sp1,mp2,mp,sp3,sp2,p,mp3)
\tkzDefMidPoint(msp3,p)
\tkzGetPoint{x1}
\tkzDefMidPoint(msp3,sp2)
\tkzGetPoint{x2}
\tkzDefMidPoint(p,sp2)
\tkzGetPoint{x3}
\tkzDefMidPoint(p2,sp2)
\tkzGetPoint{x4}
\tkzDefMidPoint(msp1,sp2)
\tkzGetPoint{x5}
\tkzDefMidPoint(msp1,sp3)
\tkzGetPoint{x6}
\tkzDefMidPoint(sp3,p3)
\tkzGetPoint{x7}
\tkzDefMidPoint(mp,sp3)
\tkzGetPoint{x8}
\tkzDefMidPoint(msp2,sp3)
\tkzGetPoint{x9}

\tkzDefMidPoint(msp2,mp)
\tkzGetPoint{x10}
\tkzDefMidPoint(msp2,mp2)
\tkzGetPoint{x11}
\tkzDefMidPoint(msp2,sp1)
\tkzGetPoint{x12}

\tkzDefMidPoint(msp3,mp3)
\tkzGetPoint{x14}
\tkzDefMidPoint(msp3,sp1)
\tkzGetPoint{x13}
\tkzDrawPolygon[dgreen,thick](x1,x2,x3,x4,x5,x6,x7,x8,x9,x10,x11,x12,x13,x14)


\end{tikzpicture}
  \caption{Construction from \Cref{lem:beforePlanarCase}: $\bar K$ (yellow), $\frac{\bar K - \bar K}{2}$ (green). 
  }  
  \label{fig:two-step-trafo}
\end{figure}
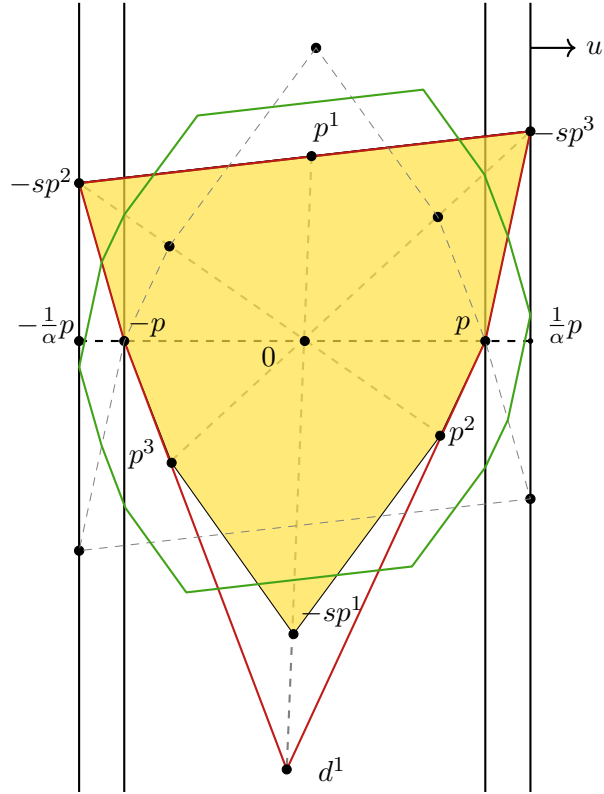

We assume w.l.o.g.~that $\aff(\{ sp^2,-sp^3\})= H^{=}_{\left(\smallmat{1 \\ 0},1\right)}$ is orthogonal to $[-p,p]$ and define $\alpha$ to be such that 
\[
\frac{1}{\alpha} p \in \bd( \conv(\bar K\cup(-\bar K))). 
\]

Since $p \in \bd(\bar K) \cap \bd(-\bar K)$ and $p \not \in \conv(\{\pm sp^1,\pm sp^2, \pm sp^3\}) = \bd \left( \conv(\bar K\cup(-\bar K))\right)$, holds $\alpha \leq 1$. Then $p=\smallmat{\alpha \\ 0}$ and following the notations from \Cref{lem:beforePlanarCase} we have
\[
\frac{1}{\alpha} p \in [sp^2,-sp^3]. 
\]


Hence, $sp^2 \in H^{=}_{\left(\smallmat{1 \\ 0},1\right)}$, $sp^3 \in H^{=}_{\left(\smallmat{-1 \\ 0},1\right)}$, and therefore the first coordinates of $p^2$ and $p^3$ are $\frac 1 s$ and $-\frac 1 s$, respectively. Thus, remembering that $p^2,p^3$ are located on the same side of $\aff\{-p,p\}$,
we may further assume that $p^2=\smallmat{1/s \\ -\nu}$ and $p^3=\smallmat{-1/s \\ -1}$ for some $\nu \in (0,1]$.

Now, we calculate the coordinates of $d^1$ as the intersection point of the lines $\aff\{p,p^2\}$ and $\aff\{-p,p^3\}$. We obtain that those coordinates satisfy the following system of equations: 

\begin{align*}
d^1_2 &= \frac{s \nu}{s \alpha-1} d^1_1 - \frac{s \alpha \nu}{s \alpha-1},\\ 
d^1_2 &= \frac{-s}{s\alpha-1} d^1_1 - \frac{s \alpha}{s\alpha-1}.
\end{align*}


Solving, 
gives us 

\[
d^1=\begin{pmatrix}
\frac{\alpha(\nu-1)}{\nu+1} \\ \frac{-2s \alpha \nu}{(s\alpha-1)(\nu+1)}  
\end{pmatrix},
\]

and for $\gamma$ such that $-\gamma d^1\in[-sp^2,-sp^3]$ holds 

\[
\frac{1}{\gamma}=\frac{2\alpha}{(\nu +1)^2}\left( \frac{2\nu }{s \alpha-1} - \frac{(\nu-1)^2}{2}\right) \geq 0. 
\]


Let $d^2$ denote the intersection point of $H^{=}_{\left(\smallmat{-1 \\ 0} , 1\right)}$ and $\aff(\set{-p,p^3})$. Hence,  $d^2=-p + \mu(-p-p^3) = \smallmat{-\alpha \\ 0} + \mu \smallmat{-\alpha+\frac 1 s \\ 1}$ for some $\mu>0$ and we directly see that $d^2_2 = \mu$. Now, since we have $d^2_1=-1$ we obtain
$d^2_2 = \frac{s(1-\alpha)}{s\alpha-1}$. Then, we have
\begin{equation*}
    d^2= \begin{pmatrix}
        -1 \\ \frac{s(1-\alpha)}{s\alpha -1}
    \end{pmatrix}. 
\end{equation*}

Similarly, let $d^3$ denote the intersection point of $H^{\le}_{\left(\smallmat{1 \\ 0} , 1\right)}$ and $\aff(\set{p,p^2})$. Hence,  $d^3=p + \mu(p-p^2) = \smallmat{\alpha \\ 0} + \mu \smallmat{\alpha-\frac 1 s \\ a}$ for some $\mu>0$. Now, since we have $d^3_1=1$ we obtain

$d^3_2 = \frac{s\nu (1-\alpha)}{s\alpha-1}$.

Then, we have

\begin{equation*}
    d^3= \begin{pmatrix}
        1 \\ \frac{s \nu(1-\alpha)}{s\alpha -1}
    \end{pmatrix}. 
\end{equation*}



Since $\nu\leq 1$, we also have 
\[
\tau^{-1}p= \begin{pmatrix}
    \tau^{-1}\alpha \\ 0
\end{pmatrix} \in \left[\frac{sp^2-sp^3}{2}, \frac{sp^2+p}{2}\right]=\left[ \begin{pmatrix}
    1 \\ \frac{s(1-\nu)}{2}
\end{pmatrix}, \begin{pmatrix}
    \frac{1+\alpha}{2} \\ \frac{-s \nu}{2}
\end{pmatrix}\right]. 
\]
Thus, 
\begin{align*}
    2\tau^{-1}\alpha &= 2\lambda +(1-\lambda)(1+\alpha),\\
    0&=\lambda s(1-\nu) +(1-\lambda)(-s \nu)
\end{align*}
for some $\lambda\in[0,1]$.

From the second equation $\lambda=\nu$, and inserting this into the first, we have
\begin{align}\label{eq:alpha_tau}
    &2\tau^{-1}\alpha = 2\nu +(1-\nu)(1+\alpha) \nonumber \\
    \iff \quad & \alpha= \frac{1+\nu}{2\tau^{-1}+\nu-1}. 
\end{align}
Note that $2\tau^{-1}+\nu-1 >0$ and $\alpha \geq 0$.


Hence,
\[
\frac{1}{\gamma}=\frac{2}{(\nu+1)(2 \tau^{-1} +\nu-1)}\left( \frac{2 \nu}{s \left(\frac{\nu+1}{2 \tau^{-1} +\nu-1}  \right)-1} - \frac{(\nu-1)^2}{2}\right).
\]




Observe that from the convexity of $\bar K$ 
we directly obtain 
\begin{align}
    -sp^1 &\in\left[0,d^1\right], \\
    d^2 &\in\left[-\frac{1}{\alpha}p,-sp^2\right], \quad \text{and}\\
    d^3 &\in\left[\frac{1}{\alpha}p,-sp^3\right] \label{eq:d3},  
\end{align}
which is equivalent to 

\begin{align*}
    s  \gamma &\leq 1, \\
    s \nu &\geq \frac{s(1-\alpha)}{s\alpha-1}, \quad \text{and}\\
    s &\geq \frac{s \nu(1-\alpha)}{s\alpha-1}.    
\end{align*}


Applying \eqref{eq:alpha_tau}, we see that $s \geq \frac{s\nu(1-\alpha)}{s\alpha-1}$ becomes equivalent to 
$\tau \geq \frac{2}{s+1}$, which by \cite[Theorem 1.7]{BDG1} is always fulfilled, and thus \eqref{eq:d3} is always fulfilled and the system of inequalities form above can be rewritten as 
\begin{equation}\label{eq:condition}
    \frac{2\tau^{-1}+\nu-2}{\nu} \leq s \leq \frac{2}{(\nu+1)(2 \tau^{-1} +\nu-1)}\left( \frac{2 \nu}{s \left(\frac{\nu+1}{2 \tau^{-1} +\nu-1}  \right)-1} - \frac{(\nu-1)^2}{2}\right),
\end{equation}
where the right hand side follows from the first and the left hand side from the second inequality. 



We want to characterize the situation, in which $s$ becomes maximal, under this condition. 


\
The right-hand side inequality of \eqref{eq:condition} can be transformed into
\[
s^2-\frac{1}{\nu+1} \left(
\frac{4 \tau^{-1}(\tau^{-1} +\nu-1)}{2\tau^{-1} +\nu-1}
\right)s-1 \le 0.
\]


Thus, we are interested in the bigger root of the quadratic polynomial, i.e., in 
\begin{equation}
   \label{eq:s_taua} 
s=\frac{\frac{4 \tau^{-1}(\tau^{-1} +\nu-1)}{(\nu+1)(2\tau^{-1} +\nu-1)}+\sqrt{\left(\frac{4 \tau^{-1}(\tau^{-1} +\nu-1)}{(\nu+1)(2\tau^{-1} +\nu-1)}\right)^2+4}}{2}=:s(\tau, \nu). 
\end{equation}

Let $\nu^+(\tau):= 1-\tau^{-1}+\sqrt{\tau^{-1}(2-\tau^{-1})}$ and $B(\tau,\nu):= \frac{4 \tau^{-1}(\tau^{-1} +\nu-1)}{(\nu+1)(2\tau^{-1} +\nu-1)}$. 

Since 
\begin{equation}
 \frac{d}{dB} \frac{B+\sqrt{B^2+4}}{2}= \frac{1}{2}\left(1+ \frac{B}{\sqrt{B^2+4}}\right)>0,
\end{equation}

we see that if $B(\tau,\nu))$ is decreasing (increasing) in $\nu$, then $s(\tau,\nu)$ is decreasing (increasing) in $\nu$, respectively. Now,
\begin{align*}
    \frac{d}{d\nu} B(\tau,\nu)&= \frac{d}{d\nu} \frac{4 \tau^{-1}(\tau^{-1} +\nu-1)}{(\nu+1)(2\tau^{-1} +\nu-1)} \\
&= \frac{-4\tau^{-1}(\nu^2+\nu(2\tau^{-1}-2)+2\tau^{-2}-4\tau^{-1}+1}{(\nu+1)^2(2\tau^{-1}+\nu-1)^2}\\
&=\frac{-4\tau^{-1}(\nu- (1-\tau^{-1}-\sqrt{\tau^{-1}(2-\tau^{-1})}))(\nu- (1-\tau^{-1}+\sqrt{\tau^{-1}(2-\tau^{-1})}))}{(\nu+1)^2(2\tau^{-1}+\nu-1)^2}
\end{align*}

and since $1-\tau^{-1}-\sqrt{\tau^{-1}(2-\tau^{-1})}\leq 0$ and $1-\tau^{-1}+\sqrt{\tau^{-1}(2-\tau^{-1})}\in[0,1]$ for $\tau\in[\tfrac{2}{3},1]$, $s(\tau,\nu)$ is increasing in $(0,\nu^+(\tau)]$, decreasing in $[\nu^+(\tau),1]$ and attains its maximum at $\nu^+(\tau)$. 

Inserting \eqref{eq:s_taua} into the left-hand-side inequality of \eqref{eq:condition}, yields to 
\begin{equation}\label{eq:tau_new0}
    \nu \geq \sqrt{(\tau^{-1}-1)(2\tau^{-1}-1)} =:\nu^*(\tau).
\end{equation}
Thus, in the case when $\nu^*(\tau)> \nu^+(\tau)$, the maximal asymmetry we can reach is $s(\tau,\nu^*(\tau))$, otherwise it is $s(\tau,\nu^+(\tau))$.

Thus, combining those two facts we see that 
the maximal asymmetry $s$ we can reach in dependence of $\tau$ is 
\begin{equation}
  s_{\max}(\tau):= s\left(\tau, \max\set{\nu^*(\tau),\nu^+(\tau)}\right). 
\end{equation}
Let $\hat \tau $ be such that  $\nu^*(\hat \tau)=\nu^+(\hat \tau)$ and $\hat s := s(\hat \tau, \nu^+(\hat \tau))$. Note that $\hat \tau \approx 0.78$ and $\hat s \approx 1.85$.

Observe that in case, when $\tau \in [\hat \tau, 1]$, we have $s(\tau,\nu(\tau))\leq \hat s$, i.e., $ s_{\max}(\tau)=s\left(\tau, \nu^+(\tau) \right)$ and thus, 
\begin{align}\label{eq:tau_new2}
    \tau &\leq \frac{(s^2+1)^2}{(s^2-1)\left(s^2+2s-1 +2\sqrt{s(s^2-1)}\right)}, 
\end{align}
while in case, when $\tau \in \left[\frac23, \hat \tau\right]$, we have $s(\tau,\nu(\tau))\geq \hat s$, i.e., $ s_{\max}(\tau)=s\left(\tau, \nu^*(\tau) \right)$, and thus, 
\begin{equation}\label{eq:tau_new1}
    \tau \leq \frac{2(s^2-2s-1)}{(s-3)(s+1)}.
\end{equation}
\end{proof}

In the following we deal with the equality case of \Cref{lem:PlanarCase}. 
\begin{lemma}\label{lem:PlanarCase_equality}
Let $K\in\mathcal K^2$ be Minkowski centered such that $s(K)>\varphi$ and $\tau(K)=c(s(K))$. 
Then, following the notations of \Cref{lem:PlanarCase}, 
\begin{enumerate}[(i)]
  \item $\conv(\{ -s( K)p^2, -s( K)p^3, p, d^1, -p \}) \subset K \subset \conv(\{ -s( K)p^2, -s( K)p^3, d^3, d^1, d^2\})$, when $\varphi < s(K) \leq \hat s$, 
  \item $\conv(\{ -s( K)p^2, -s( K)p^3, p, d^1\}) \subset K \subset \conv(\{ -s( K)p^2, -s( K)p^3, d^3, d^1\})$, when $\hat s < s( K) < 2$, and 
  \item $K= \conv(\{ -s( K)p^1, -s( K)p^2, -s( K)p^3 \})$, when $s( K) = 2$.
\end{enumerate}
\end{lemma}

\begin{proof} Let $K\in\mathcal K^2$ be Minkowski centered with $s(K)>\varphi$ and define $\bar K$ as in \Cref{lem:beforePlanarCase} with $\tau(\bar K) = c(s(\bar K))$. By \Cref{lem:PlanarCase} we have the following: 

For $\varphi <s(\bar K) \leq \hat s$ we have $\tau(\bar K)= c(s(\bar K))= \frac{(s(\bar K)^2+1)^2}{(s(\bar K)^2-1)\left(s(\bar K)^2+2s(\bar K)-1 +2\sqrt{s(\bar K)(s(\bar K)^2-1)}\right)}$, which implies equality in the right-hand-side inequality of \eqref{eq:condition}, i.e., $d^1=-s(\bar K)p^1$, and 
\[ 
\bar K= \conv(\{ -s(\bar K)p^2, -s(\bar K)p^3, p, d^1, -p \}). 
\]

For $\hat s < s(\bar K) < 2$ we have $\tau(\bar K)= c(s(\bar K)= \frac{2(s(\bar K)^2-2s(\bar K)-1)}{(s(\bar K)-3)(s(\bar K)+1)} $.
Observe that this implies equality in \eqref{eq:tau_new0}, and thus also in both inequalities of \eqref{eq:condition}, i.e., $d^1=-s(\bar K)p^1, d^2 =-s(\bar K)p^2$, and therefore, $\bar K= \conv(\{ -s(\bar K)p^2, -s(\bar K)p^3, p, d^1\})$. 

For $s(\bar K)=2$, in addition to $d^1=-s(\bar K)p^1, d^2 =-s(\bar K)p^2$, we also have $d^3= -s(\bar K)p^3$ and $\bar K$ becomes $\conv(\{ -s(\bar K)p^1, -s(\bar K)p^2, -s(\bar K)p^3\})$. 
 
Thus, we have 
\begin{enumerate}[(i)]
  \item $\bar K= \conv(\{ -s(\bar K)p^2, -s(\bar K)p^3, p, d^1, -p \})$, when $\varphi < s(\bar K) \leq \hat s$, 
  \item  $\bar K= \conv(\{ -s(\bar K)p^2, -s(\bar K)p^3, p, d^1\})$, when $\hat s < s(\bar K) < 2$, and 
  \item $\bar K= \conv(\{ -s(\bar K)p^2, -s(\bar K)p^3, d^1 \})$, when $s(\bar K) = 2$.
\end{enumerate}

Next, we show that the only Minkowski centered convex body $K'$ of the form \eqref{eq:K'} with $s(K') > \varphi$ and $\tau(K')=c(s(K'))$ are still the same as in (i)-(iii) from above. 

Notice that before Step 4 of \Cref{lem:beforePlanarCase}b), both segments $[-sp^2,-p]$ and $[-sp^3,p]$ already belonged to $\bd(K'_3)$ and 
$p^2 \in [p, -sp^1]$, $p^3 \in [-p, -sp^1]$,
implying $[-p, -sp^1], [p, -sp^1] \subset \bd(K'_3)$. Since we also had $p^1 \in [-sp^2, -sp^3]$, we conclude that before Step 4 $K'_3$ must also be as in (i)-(iii). For the same reasoning, we conclude that before Step 3 $K'_2$ could not be different. 

If before Step 2 $K'_1$ would be different from (i)-(iii), this would mean that at least one of $-sp^3$, $-sp^2$ was modified during this step. This would imply 
\[
\tau^{-1}(K'_2)=\rho_2(p)< \rho_1(p)=\tau^{-1}(K'_1), 
\]
a contradiction. Thus, we conclude that before Step 2 $K'_1$ must have been be as in (i)-(iii).

Now, if before Step 1 $K'$ would be different from those in (i)-(iii), before this step at least one of $-sp^i$, $i=1,2,3$ must have been equal to $-\gamma_i^{-1} sp^i$ with $\gamma_i <1$. 

Observe that since $p^2 \in [-p, -sp^1]$, $p^3 \in [p, -sp^1]$ and $p^1 \in [-sp^2, -sp^3]$, this could be only possible if all three points $-sp^i$, $i=1,2,3$ have been scaled during this step. But in this case, $\tau(K') <\tau(K'_1)$, implying that even before Step 1 $K'$ must have been as in (i)-(iii).

We check the shape of $K$ before one applies transformation \eqref{eq:K'}. Note that $s(K)=s(\bar K)$. 


If $\varphi < s(K) \leq \hat s$, the only possible freedom we have in choosing the original set $K$ is to replace the linear boundary $[p,-sp^3]$, s.t.~$\bd(K) \cap \pos(\{-sp^3,p\}) \subset \conv(\{-sp^3,p, d^3\})$. 

If $\hat s < s(K) <  s$, possible freedom we have in choosing the original set $K$ is to replace the linear boundaries $[p,-sp^3]$ and $[-p-sp^2]$, s.t.~$\bd(K) \cap \pos(\{-sp^3,p\}) \subset \conv(\{-sp^3,p, d^3\})$ and $\bd(K) \cap \pos(\{-sp^2,-p\}) \subset \conv(\{-sp^2,-p, d^2\})$, respectively.


Obviously, for $s(K)=2$, $K$ must be a triangle and there is no freedom left. 
\end{proof} 

Now, we are ready to prove \Cref{thm:PlanarCase}. 

\begin{proof}[Proof of \Cref{thm:PlanarCase}]
In \cite[Theorem 1.7]{BDG1} it is shown that $\frac{2}{s(K)+1} \leq \tau(K) \leq 1$. 

Combining this with \Cref{lem:PlanarCase}, we obtain 
\[
\frac{2}{s(K)+1} \leq \tau(K) \leq c(s(K)).
\]

Now we show that for any $s\in[1,2]$ and $\tau\in\left[\frac{2}{s+1},c(s)\right]$, where 
\[ 
c(s) :=
    \begin{dcases}
     1 & s \leq \varphi, \\
    \frac{(s^2+1)^2}{(s^2-1)\left(s^2+2s-1 +2\sqrt{s(s^2-1)}\right)} & \varphi <s \leq \hat s, \\
    \frac{2(s^2-2s-1)}{(s-3)(s+1)}  & \hat s < s \leq 2.   \\
    \end{dcases}
\]
there exists $K_{s,\tau}\in\mathcal K^2$, such that $s(K_{s,\tau})=s$ and $\tau(K_{s,\tau})=\tau$. 
In order to simplify the proof we split the region of possible $(s, \tau)$ into two parts: the first one being $s\in[1,2]$ and $\tau\in\left[\frac{2}{s+1},\min\set{1,\frac{s}{s^2-1}}\right]$, while the second being $s \in [\varphi,2]$ and $\tau\in\left[\frac{s}{s^2-1}, c(s)\right]$.  

\textbf{Step 1:}
First, we show that for any $s\in[1,2]$ and $\tau\in\left[\frac{2}{s+1},\min\set{1,\frac{s}{s^2-1}}\right]$ 
there exists $K_{s,\tau}\in\mathcal K^2$, such that $s(K_{s,\tau})=s$ and $\tau(K_{s,\tau})=\tau$. Let $S=\conv\left(\set{p^1,p^2,p^3}\right)$ be a regular Minkowski centered triangle and $K_s:= S \cap (-sS)$, $s\in[1,2]$. By \cite[Remark 4.1]{BDG1}, $K_s$ is Minkowski centered, $s(K_s)=s$ and $\tau(K_s)=\alpha(K_s)=\frac{2}{s+1}$. Let $q^i$, $i=2,3$, be the vertices of $K_s$, that are the intersection points of the edges of $K_s$ with the normal vectors $\frac{s}{2}p^i$ and $-\frac{1}{2}p^1$, respectively. We define a continuous transformation $f:\{K_s: s\in[1,2]\} \times [0,1] \rightarrow \mathcal K^2$ with $s(f(K_s,t)=s$ for all $t\in[0,1]$, while $f(K_s,0)=K_s$ and $\tau(f(K_s,1))=\min\set{1,\frac{s}{s^2-1}}$. 

For $t\in [0,\frac{1}{2}]$, we continuously rotate the supporting lines that contain $\frac{s}{2}p^i$ around $q^i$, $i=2,3$, respectively, until the new edges are orthogonal to the edge containing $-\frac{1}{2}p^1$. 

For $t\in [\frac{1}{2},1]$, we rotate the lines that contain the edges of $K_s$, that contain $-\frac{1}{s}q^i$, $i=1,2$, around those points, respectively, until they intersect in $\frac{s}{2}p^1$. 

\begin{figure}
    \centering
    \begin{subfigure}{0.3\textwidth}
    \centering
    \begin{tikzpicture}[scale=2.5]

    \draw [thick, gray] (0,1)--(0.87,-0.5)--(-0.87,-0.5)--(0,1); 
    \draw [thick, gray, rotate around={180: (0,-0.01)}] (0,1)--(0.87,-0.5)--(-0.87,-0.5)--(0,1); 
   
    \draw [fill] (0,-0.01) circle [radius=0.01];

    \draw [thick, gray, dashed] (-0.6348,-0.5) -- (0.405,0.3);
    \draw [thick, gray, dashed] (0.6348,-0.5) -- (-0.405,0.3);
    \draw [thick, gray, dashed] (0,-0.5) -- (0,0.8);

    \draw [thick,  fill=lgold, fill opacity=0.7] (-0.11547,0.8) -- (0.11547,0.8) -- (0.75038,-0.2997)-- (0.6348,-0.5)-- (-0.6348,-0.5)-- (-0.75038,-0.2997)--(-0.11547,0.8);
    
    \draw [fill] (0.64,-0.5) circle [radius=0.02];
    \draw [fill] (-0.64,-0.5) circle [radius=0.02];

    \draw [fill, dred] (0,-0.5) circle [radius=0.02]; 
    \draw [fill, dred] (0.405,0.3) circle [radius=0.02]; 
    \draw [fill, dred] (-0.405,0.3) circle [radius=0.02]; 
    
    \draw (-0.07,-0.07) node {$0$};
    \draw (0,1.1) node {$p^1$};
    \draw (0.9,-0.65) node {$p^2$};
    \draw (-1,-0.55) node {$p^3$};
    \draw (0.6,-0.65) node {$q^2$};
    \draw (-0.6,-0.65) node {$q^3$};
    \draw (0,-0.65) node {$-\frac 12 p^1$};
    \end{tikzpicture}
    \end{subfigure}
    \hfill
    \begin{subfigure}{0.3\textwidth}
    \centering
    \begin{tikzpicture}[scale=2.5]
    \draw [thick, gray] (0,1)--(0.87,-0.5)--(-0.87,-0.5)--(0,1); 
    \draw[thick, gray, rotate around={180: (0,-0.01)}] (0,1)--(0.87,-0.5)--(-0.87,-0.5)--(0,1); 
   
    \draw [fill] (0,-0.01) circle [radius=0.01];

    \draw [thick, gray, dashed] (-0.6348,-0.5) -- (0.405,0.3);
    \draw [thick, gray, dashed] (0.6348,-0.5) -- (-0.405,0.3);
    \draw [thick, gray, dashed] (0,-0.5) -- (0,0.8);
        
    \draw [thick, fill=lgold, fill opacity=0.7] (-0.11547,0.8) -- (0.11547,0.8) -- (0.6348,-0.1)-- (0.6348,-0.5)-- (-0.6348,-0.5)-- (-0.6348,-0.1)--(-0.11547,0.8);
    \draw (-0.07,-0.07) node {$0$};

    \draw [fill, dred] (0,-0.5) circle [radius=0.02]; 
    \draw [fill, dred] (0.405,0.3) circle [radius=0.02]; 
    \draw [fill, dred] (-0.405,0.3) circle [radius=0.02]; 

    \draw (0.59,0.35) node {$-\frac 1s q^3$};
    \draw (-0.57,0.35) node {$-\frac 1s q^2$};
    \end{tikzpicture}
    \end{subfigure}
    \hfill
    \begin{subfigure}{0.3\textwidth}
    \centering
    \begin{tikzpicture}[scale=2.5]
    \draw [thick, gray] (0,1)--(0.87,-0.5)--(-0.87,-0.5)--(0,1); 
    \draw [thick, gray, rotate around={180: (0,-0.01)}] (0,1)--(0.87,-0.5)--(-0.87,-0.5)--(0,1);

    \draw [thick, gray, dashed] (-0.6348,-0.5) -- (0.405,0.3);
    \draw [thick, gray, dashed] (0.6348,-0.5) -- (-0.405,0.3);
    \draw [thick, gray, dashed] (0,-0.5) -- (0,0.8);

    \draw [fill] (0,-0.01) circle [radius=0.01];

    \draw [thick, fill=lgold, fill opacity=0.7] (0,0.8) -- (0.6348,0.03)-- (0.6348,-0.5)-- (-0.6348,-0.5)-- (-0.6348,0.03)--(0,0.8);

    \draw [fill, dred] (0,-0.5) circle [radius=0.02]; 
    \draw [fill, dred] (0.405,0.3) circle [radius=0.02]; 
    \draw [fill, dred] (-0.405,0.3) circle [radius=0.02]; 
    \end{tikzpicture}
    \end{subfigure}
    \caption{Transformation within the proof of \Cref{thm:PlanarCase}: regular Minkowski centered triangle $S$ and $-S$ (gray), $f(K_s,1)=K_s = S \cap (-s S)$ before the transformations (Subfigure 6.1), the transformed set $f(K_s,\frac{1}{2})$ (Subfigure 6.2), the transformed set $f(K_s,1)$ (Subfigure 6.3), the asymmetry points $-\frac 12 p^1, -\frac 1s q^2, -\frac 1s q^3$ (big black dots) of $f(K_s,t)$, $t\in[0,1]$. 
}
\label{fig:ScapsS}
\end{figure}
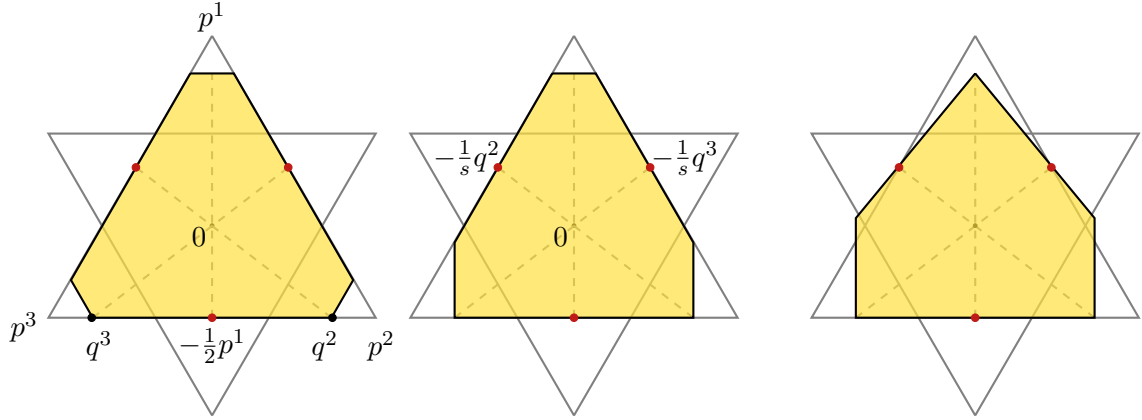

It is easy to verify that at every point of this transformation, $-\frac{1}{s}K_s\optc K_s$, with $-\frac{1}{2}p^1,-\frac{s}{2}q^2,-\frac{s}{2}q^2$ being a triple of well-spread asymmetry points. Hence, $s(f(K_s,t))=s$ for every $t\in[0,1]$ and $s\in[1,2]$. 

Now consider $f(K_s,1)$. One can compute
\begin{equation*}
    f(K_s,1)= \conv\left( \set{\frac{s}{2}p^1, q^2,q^3, \begin{smallpmatrix}
        \frac{1}{2\sqrt{3}}(2s-1) \\ \frac{-s^2+s+1}{2}
    \end{smallpmatrix},\begin{smallpmatrix}
        -\frac{1}{2\sqrt{3}}(2s-1) \\ \frac{-s^2+s+1}{2}
    \end{smallpmatrix}} \right).
\end{equation*}
Thus, for $s\in[0,\varphi]$, $H^=_{\left(\smallmat{\pm 1 \\ 0} , 1\right)}$ supports $f(K_s,1)$ at $\begin{smallpmatrix}
        \pm\frac{1}{2\sqrt{3}}(2s-1) \\ 0
\end{smallpmatrix}$, respectively. Therefore, by \cite[Theorem 3]{BDG}, $\tau(f(K_s,1))=1$, while for $s\in(\varphi,2]$, the body $f(K_s,1)$ is of the form $K'_4$ from the proof of \Cref{lem:beforePlanarCase} and $\tau$ is attained in the horizontal direction with $\tau(f(K_s,1))=\frac{s}{s^2-1}$. 

Since the transformation is continuous, we conclude that 
\begin{equation*}
   \set{\tau(f(K_s,t)):t\in [0,1]}=\left[\frac{2}{s+1},\min \set{1,\frac{s}{s^2-1}}\right] 
\end{equation*}

for every $s\in[1,2]$, as desired. 

\textbf{Step 2:}
For the second part we show that for every $s \in [\varphi,2]$ and $\tau\in\left[\frac{s}{s^2-1}, c(s)\right]$ 
there exists a Minkowski centered convex body $K$ with $s(K)=s$, $\tau(K)=\tau$. Since $\frac{2}{s(K)+1} \leq \tau(K) \leq 1$ and $s \in [\varphi,2]$, we have $\tau\in\left[\frac{2}{3},1\right]$. Moreover, $\tau=\frac{s}{s^2-1}$ rewrites as $s=\frac{1}{2\tau}+\sqrt{1+\frac{1}{4\tau^2}}$. Thus, the statement above is equivalent to showing that for every $\tau\in\left[\frac{2}{3},1\right]$ and $s\in\left[\frac{1}{2\tau}+\sqrt{1+\frac{1}{4\tau^2}} , s_{\max}(\tau)\right]= \left[\frac{1}{2\tau}+\sqrt{1+\frac{1}{4\tau^2}} , s\left(\tau, \max\set{\nu^*(\tau),\nu^+(\tau)}\right) \right]$ there exists a Minkowski centered convex body $K$ with $s(K)=s$ and $\tau(K)=\tau$. 

Here we use the notations introduced in the proof of \Cref{lem:PlanarCase}. Consider for every $\nu\in[\max\set{\nu^+(\tau),\nu^*(\tau)},1]$ 
\begin{equation}
    K(\tau,\nu):= \conv\left(\set{-s(\tau,a)p^1,-s(\tau,\nu)p^3,p,-p,d^1}\right). 
\end{equation}
We know by construction that $K(\tau,a)$ is Minkowski centered with asymmetry $s(\tau,a)$ and $\tau(K(\tau,\nu))=\tau$. 
Since 
\begin{equation*}
    s\left(\tau, \left[\max\set{\nu^+(\tau),\nu^*(\tau)},1\right]\right)=\left[\frac{1}{2\tau}+\sqrt{1+\frac{1}{4\tau^2}},s_{\max}(\tau)\right], 
\end{equation*}
we have shown that there exist bodies with the desired properties for every $\tau\in\left[\frac{2}{3},1\right]$. 
\end{proof}

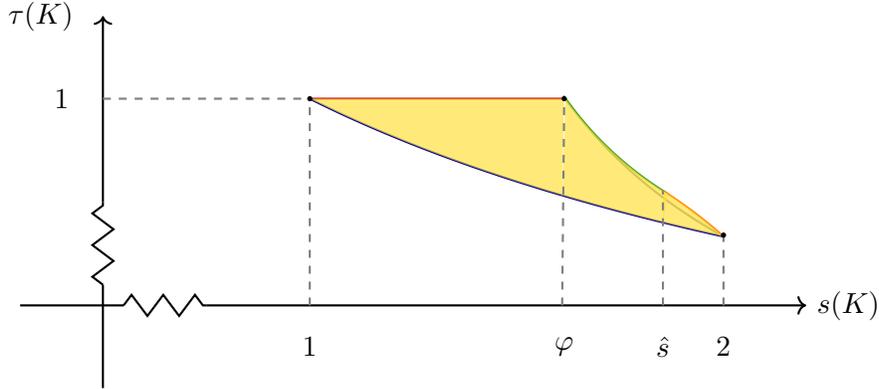
\begin{figure}[ht]
  \begin{tikzpicture}[scale=5.5, 
  declare function={aplus(\t)= 1-\t +sqrt(\t*(2-\t)); 
                   astar(\t)=sqrt(2*(\t)^2-3*\t+1);
                   btaua(\a,\t)=(4*(\t)*(\t+\a-1))/((\a+1)*(2*(\t)+\a-1));
                   smax(\b)= 0.5*(\b+sqrt((\b)^2+4));
                   stau1(\y)=smax(btaua(aplus(1/(\y)),1/(\y)));
                   stau2(\y)=smax(btaua(astar(1/\y),1/\y)));
                  slow(\y)=2/\y -1;
                 },
  ]
    \draw[thick, discont] (0.05,0) -- (0.25,0);
    \draw[thick, discont] (0,0.05) -- (0,0.25);
    \draw [thick] (-0.2,0) -- (0.05,0);
    \draw [thick] (0,-0.2) -- (0,0.05);
    \draw[->] [thick] (0.25, 0) -- (1.7, 0) node[right] {$s(K)$};
    \draw[->] [thick] (0, 0.25) -- (0, 0.7);

    \draw [name path=F4,thick, dred,shift={(-0.5,-0.5)}] (1,1) -- (1.615,1);
    \draw[name path=F1, thick, dblue, domain=1:1.61, smooth, variable=\x, dblue,shift={(-0.5,-0.5)}]  plot ({\x}, {2/(\x+1)});
    \draw[name path=F2,thick, dblue, domain=1.61:1.85364, smooth, variable=\x, dblue,shift={(-0.5,-0.5)}]  plot ({\x}, {2/(\x+1)});
    \draw[name path=F3,thick, dblue, domain=1.85364:2, smooth, variable=\x, dblue,shift={(-0.5,-0.5)}]  plot ({\x}, {2/(\x+1)});
    \draw[ thick, domain=1.613:2,gray , variable=\x,shift={(-0.5,-0.5)}]  plot ({\x}, {(\x)/((\x)^2-1)});
    \draw[name path=F5, domain=0.77:1.000, smooth, variable=\y,shift={(-0.5,-0.5)},thick,dgreen] plot ({stau1(\y)}, {\y});
   \draw[name path=F6, domain=0.66666: 0.77724, smooth, variable=\y,shift={(-0.5,-0.5)},thick, orange] plot ({stau2(\y)}, {\y});
   \tikzfillbetween[of=F1 and F4,on layer=main]{lgold, opacity=0.7};
    \tikzfillbetween[of=F2 and F5,on layer=main]{lgold, opacity=0.7};
    \tikzfillbetween[of=F3 and F6,on layer=main]{lgold, opacity=0.7};
    \draw [thick, dashed,gray] (0,0.5) -- (0.5,0.5);
    \draw [thick, dashed,,gray] (1.11,0)--(1.115,0.5);
    \draw [thick, dashed,gray] (1.5,0) -- (1.5,0.17);
    \draw [thick, dashed,gray] (0.5,0) -- (0.5,0.5);
     \draw [thick, dashed,gray,shift={(-0.5,-0.5)}] (1.85364,0.5) -- (1.85364,0.77724);
     
    \draw [fill,shift={(-0.5,-0.5)}] (1,1) circle [radius=0.005];
    \draw [fill,shift={(-0.5,-0.5)}] (1.615,1) circle [radius=0.005];
    \draw [fill,shift={(-0.5,-0.5)}] (2,0.67) circle [radius=0.005];
    \draw (-0.15,0.7) node {$\tau(K)$};
    \draw (-0.1,0.5) node {$1$};
    \draw (0.5,-0.1) node {$1$};
    \draw (1.115,-0.1) node {$\varphi$};%
    \draw (1.5,-0.1) node {$2$};
    \draw (1.35364,-0.1) node {$\hat s$};
     \end{tikzpicture}
     \caption{Region of possible values for the parameter $\tau(K)$ for Minkowski centered $K \in\K^2$ (yellow): $\tau(K) \geq \frac{2}{s+1}$ (blue); $\tau(K) \leq c(s(K))$ (red, green and orange, resp.). 
     The black dots are given by $0$-symmetric $K \in\K^2$ ($s=1$), the Golden House $\GH$ ($s=\varphi$) and triangles ($s=2$).
      }
     \label{fig:alpha-region}
\end{figure}

In \cite{BDG,BDG1,BDG2} two additional means of convex bodies and their relations are considered. The harmonic mean of $K\in\K^n$ is defined as $\left( \frac{K^{\circ}-K^{\circ}}{2} \right)^{\circ}$ and the maximum as $\conv\left( K \cup (-K) \right)$. Understanding their containment factors helps for instance, to prove new inequalities for diameter variants \cite{BR_diametervariants}. For $K \in \K^n$ we define $\gamma(K)>0$ such that 
\[
\left( \frac{K^{\circ}-K^{\circ}}{2} \right)^{\circ} \subset^{\opt} \gamma(K) \, \conv\left( K \cup (-K) \right).
\]
As with $\alpha(K)$ and $\tau(K)$, $\gamma(K)$ does not only depend on the Minkowski asymmetry of $K$. Using polarity and \Cref{thm:PlanarCase}, we are able to show the following.
\begin{cor}
    \label{cor:GammaPlanarCase}
 
Let $K \in\K^2$ be Minkowski centered. Then 
\[
\frac{2}{s(K)+1} \leq \gamma(K) \leq c(s(K)).
\] 
Moreover, for every pair
$(\gamma,s)$, such that $1 \leq s \leq 2$ and $\frac{2}{s+1} \leq \gamma\leq c(s(K))$, there exists a Minkowski centered $K \in\K^2$, such that $s(K)=s$ and $\gamma(K)=\gamma$.
\end{cor}

\begin{proof}[Proof of \Cref{cor:GammaPlanarCase}]

We may assume w.l.o.g. $s(K) \neq 1$. Thus, by \Cref{cor:zero-inside}, there exists a triple of well-spread asymmetry points of $K$ that contain zero in their convex hull. 
Thus, $K \optc -s(K)K$ implies $-K^{\circ} \optc s(K)K^{\circ} $, i.e., $s(K^{\circ})=s(K)$ and $K^{\circ}$ is Minkowski-centered. 
    
Now, on the one hand, polarising 
    \begin{equation*}
        K^{\circ} \cap (-K^{\circ}) \optc\tau(K^{\circ})\frac{K^{\circ}-K^{\circ}}{2}, 
    \end{equation*}  
    we obtain
   \begin{equation*}
        \frac{1}{\tau(K^{\circ})} \left( \frac{K^{\circ}-K^{\circ}}{2} \right)^{\circ} \subset \left(K^{\circ} \cap (-K^{\circ})\right)^{\circ}= \conv\left(K\cup (-K) \right),  
    \end{equation*}
and hence, $ \gamma(K)\leq \tau(K^{\circ})$. While on the other hand, polarising 
\begin{equation*}
\left( \frac{K^{\circ}-K^{\circ}}{2} \right)^{\circ} \optc \gamma(K) \conv \left(K\cup (-K) \right)
\end{equation*} implies $\gamma(K)\geq \tau(K^{\circ})$. Thus, by \Cref{thm:PlanarCase}, 
 
\begin{equation*}
 \frac{2}{s(K)+1}  \leq \gamma(K)=  \tau(K^{\circ})\leq c(s(K^{\circ})) = c(s(K)). 
\end{equation*} 

Finally, for the bodies $K_{s,\tau}$ described in \Cref{thm:PlanarCase}, $s(K_{s,\tau}^{\circ})=s$ and $\gamma(K_{s,\tau}^{\circ})=\tau$,  showing that for every pair
$(\gamma,s)$, such that $1 \leq s \leq 2$ and $\frac{2}{s+1} \leq \gamma\leq c(s(K))$, there exists a Minkowski centered $K \in\K^2$, such that $s(K)=s$ and $\gamma(K)=\gamma$.
\end{proof}

\section{Diameter-width-ratio for (pseudo-)complete sets}

We recall the characterization of pseudo-completeness from \cite{BrG2}. 
 \begin{proposition}\label{prop:pseudo-complete-charact}
  Let $K \in \K^n$ and $C \in \K_0^n$. Then the following are equivalent:
 \begin{enumerate}[(i)]
 \item $K$ is pseudo-complete \wrt~$C$,
  \item $(s(K)+1) r(K,C)= r(K,C)+R(K,C) = \frac{s(K)+1}{s(K)} R(K,C)=  D(K,C)$, and
  \item for every incenter $c$ of $K$ w.r.t. $C$ we have
   \[\frac{s(K)+1}{2s(K)}(-(K-c)) \subset \frac{K-K}2 \subset \frac12 D(K,C) C \subset \frac{s(K)+1}2 (K-c)\]
  is left-to-right optimal, which implies that $c$ is also a circumcenter and a Minkowski center of $K$.
  \end{enumerate}
  \end{proposition}


We show for the diameter-width ratio bound for pseudo-complete sets in the planar case. 
\begin{proof}[Proof of \Cref{thm:results_pscomp}]
We assume w.l.o.g.~that $C\optc K$. Thus, $r(K,C)=1$ and by \Cref{prop:pseudo-complete-charact}, $K$ is Minkowski centered.
Abbreviating $s:=s(K)$ again, we obtain $D(K,C)=(s+1)r(K,C)=s+1$ and 
\begin{equation*}
 \frac{K-K}{2} \subset \frac{D(K,C)}{2} C = \frac{s+1}{2} C \subset \frac{s+1}2 K \cap (-K)
\end{equation*}
from \Cref{prop:pseudo-complete-charact}.
Thus, $C \subset K \cap (-K)$, which implies $w(K,C) \geq w(K,K \cap (-K))$ and since $w(K,K \cap (-K)) = w\left(\frac{K-K}{2},K \cap (-K)\right) = 2r\left(\frac{K-K}{2},K \cap (-K)\right)$ (see \cite{Sch} for basic properties of the width)
\begin{equation} \label{eq:Dw_chain}
    \frac{D(K,C)}{w(K,C)} =\frac{s+1}{w(K,C)} \leq \frac{s+1}{2 r\left(\frac{K-K}2, K \cap (-K)\right)} 
    = \frac{s+1}2 \tau(K). 
\end{equation}
Since by \Cref{thm:PlanarCase} $\tau(K) \leq c(s(K)),$ we obtain
\[
\frac{D(K,C)}{w(K,C)} \leq \frac{s(K)+1}{2} c(s(K)).
\]

One can observe that 
\[ 
\frac{s+1}{2} c(s) =
    \begin{dcases}
    \frac{s+1}{2} & s \leq \varphi, \\
    \frac{(s^2+1)^2}{2(s-1)\left(s^2+2s-1 +2\sqrt{s(s^2-1)}\right)} & \varphi <s \leq \hat s, \\
    \frac{s^2-2s-1}{s-3}  & \hat s < s \leq 2  \\
    \end{dcases}
\]
attains its maximum if and only if $s=\varphi$. Indeed, $\frac{s+1}{2}$ is increasing for all $s$, while for $\varphi <s \leq \hat s$ holds that the derivative 
\[ 
\left(\frac{(s^2+1)^2}{2(s-1)\left(s^2+2s-1 +2\sqrt{s(s^2-1)}\right)} \right)'= \frac{(s^2-2s-1)(s^2+s-\sqrt{s(s^2-1)}}{2s(s+1)(s-1)^2} \leq 0 
\]
and for $\hat s < s \leq 2$ holds
\[ 
\left( \frac{s^2-2s-1}{s-3} \right)'= \frac{s^2-6s+7}{(s-3)^2} \leq 0. 
\]

Thus, 
\[
\frac{D(K,C)}{w(K,C)} \leq  \frac{\varphi+1}2 \approx 1.31.
\]

Now let $\tilde K, \tilde C \in \K^2$ be Minkowski centered, such that $\tilde K$ is pseudo-complete w.r.t. $\tilde C$ and 
\[
\frac{D(\tilde K,\tilde C)}{w(\tilde K,\tilde C)} = \frac{\varphi+1}2.   
\]
Then, as shown above, $s(\tilde K)= \varphi$ and we immediately see that $\tau(\tilde K)=1$. 
By \Cref{prop:old_results} this implies $\alpha(\tilde K)=1$ and thus by \cite[Theorem 1.1]{BDG}, there exists a linear transformation $L$ such that $L(K)=\mathbb{GH}$. We may choose $\tilde C=\GH\cap(-\GH)$ to make sure $\tilde K$ is pseudo-complete w.r.t. $\tilde C$. 

Combining the results, we obtain 
\[
\frac{D(K,C)}{w(K,C)} \leq \frac{s(K)+1}{2} c(s(K)) \leq \frac{D(\GH,\GH\cap(-\GH))}{w(\GH,\GH\cap(-\GH))} = \frac{\varphi+1}2 \approx 1.31.
\]

Finally, we show that for every pair $(\rho,s)$, with 
\[ 
1 \leq  \rho \leq \frac{s+1}{2} c(s)= 
    \begin{dcases}
     \frac{s+1}{2} & s \leq \varphi, \\
    \frac{(s^2+1)^2}{2(s-1)\left(s^2+2s-1 +2\sqrt{s(s^2-1)}\right)} & \varphi <s \leq \hat s, \\
    \frac{s^2-2s-1}{s-3}  & \hat s < s \leq 2,  \\
    \end{dcases}
\]
there exists some Minkowski centered $K$, s.t.~$s(K)=s$, and a set $C$, s.t.~$K$ is pseudo-complete w.r.t. $C$ and $\frac{D(K,C)}{w(K,C)}=\rho$ (c.f.~\Cref{fig:Dwratio}).

To do so, let $K_s$ be defined such that it is Minkowski centered, $s(K_s)=s \in [1,2]$ and $\tau(K_s)=c(s)$. Then define $C_\lambda=(1-\lambda) \left( \frac{K_s-K_s}{2} \right) + \lambda \frac{s+1}{2} (K_s \cap (-K_s))$ with $\lambda \in [0,1]$. This way $C_\lambda$ is a convex combination of $\frac{K_s-K_s}2$ and $\frac{s+1}{2}(K_s \cap(-K_s))$, and therefore $K_s$ is pseudo-complete w.r.t. $C_\lambda$ with $D \left(K_s,C_\lambda\right)=2$ by \Cref{prop:pseudo-complete-charact}.

We have
\begin{align*}
w(K_s,C_1) &= w \left(\frac{K_s-K_s}{2},\frac{s+1}{2} K_s \cap (-K_s)\right)= \frac{2}{s+1} w \left(\frac{K_s-K_s}{2},\frac{s+1}{2} K_s \cap (-K_s)\right) \\
&=  \frac{4}{s+1} r\left(\frac{K_s-K_s}{2},\frac{s+1}{2} K_s \cap (-K_s)\right) = \frac{4}{s+1} \frac{1}{\tau(K_s)}. 
\end{align*}
Thus, 
\[
\frac{D(K_s,C_1)}{w(K_s,C_1)}= \frac{s+1}{2} \tau(K_s).
\]
Moreover, since $\frac{D(K_s,C_0)}{w(K_s,C_0)}=1$, we have
\[
\left\{ \frac{D(K_s,C_\lambda)}{w(K_s,C_\lambda)}, 0 \leq \lambda \leq 1 \right\} 
=  \left[1,  \min \left\{ \frac{s+1}{2}, \frac{s+1}{2} c(s) \right\}  \right].
\]
\end{proof}



\bigskip

Katherina von Dichter -- 
Brandenburg University of Technology Cottbus-Senftenberg (BTU), Department of Mathematics, Germany. \\
\textbf{katherina.vondichter@b-tu.de}

Mia Runge -- 
Technical University of Munich (TUM), School of Computation, Information and Technology, Department of Mathematics, Germany. \\
\textbf{mia.runge@tum.de}

\end{document}